\documentclass[reqno]{amsart}
\usepackage{amssymb}
\usepackage{amsmath}
\usepackage{amsthm}
\usepackage{slashed}
\usepackage[foot]{amsaddr}

\newtheorem{theorem}{Theorem}

\newtheorem{lemma}[theorem]{Lemma}
\newcommand{\R}{\mathbb R}

\usepackage{geometry}
\numberwithin{theorem}{section}
\numberwithin{equation}{section}

\makeatletter
\renewcommand{\email}[2][]{%
  \ifx\emails\@empty\relax\else{\g@addto@macro\emails{,\space}}\fi%
  \@ifnotempty{#1}{\g@addto@macro\emails{\textrm{(#1)}\space}}%
  \g@addto@macro\emails{#2}%
}

\title{Linear Stability of Schwarzschild Spacetime Subject to Axial Perturbations}
\author{Pei-Ken Hung and Jordan Keller}

\address{Pei-Ken Hung, Department of Mathematics, Columbia University, Room 509, MC 4406, 2990 Broadway, New York, NY, 10027}
\email{pkhung@math.columbia.edu}

\address{Jordan Keller, Department of Mathematics, Columbia University, Room 509, MC 4406, 2990 Broadway, New York, NY, 10027}
\email{keller@math.columbia.edu}

\begin{document}

\begin{abstract}
In this paper, we address the issue of linear stability of Schwarzschild spacetime subject to certain axisymmetric perturbations.  In particular, we prove that associated solutions to the linearized vacuum Einstein equations centered at a Schwarzschild metric, with suitably regular initial data, decay to a linearized Kerr metric.  Our method employs a complex line bundle interpretation applied to a connection-level object, allowing for direct analysis of this connection-level object by the linearized Einstein equations, in contrast with the recent breakthrough of Dafermos-Holzegel-Rodnianski.
\end{abstract}
\maketitle

\section{Introduction}

The Schwarzschild spacetimes $(\mathcal{M},g_{M})$, parametrized by mass $M$, constitute the simplest family of solutions of the vacuum Einstein equations

\begin{equation}\label{eq: vacuumEinstein}
Ric(g) = 0.
\end{equation}

Stability of the Schwarzschild metric as a solution to \eqref{eq: vacuumEinstein} is an issue of fundamental importance in the study of isolated systems, with the tenability of the black hole concept, and the bizarre attendant geometry, being contingent on an affirmative answer.  Non-linear stability in the Minkowski setting was established in the monumental work of Christodoulou and Klainerman \cite{ChristKlainerman}, yet the fully non-linear problem for the Schwarzschild spacetime remains unresolved, with little direct progress to date.  On the other hand, the linearized setting has a rich history of study, beginning with modal analysis and culminating in the recent breakthrough of Dafermos-Holzegel-Rodnianski \cite{DHR}.

Early efforts by physicists assumed a formal mode decomposition in the time and angular variables, with solutions of the form

$$ e^{i\sigma t}Z(r)Y_{slm}(\theta,\phi).$$

The study of these individual modes was begun in the work of Regge-Wheeler \cite{RW}, culminating in the papers of Vishveshwara \cite {Vishveshwara} and Zerilli \cite{Zerilli}.  Later efforts by Moncrief \cite{Moncrief} placed these two authors' approaches on firmer geometric footing, while Chandrasekhar \cite{Chandra2} demonstrated a transformation theory connecting their disparate analyses.  A definitive treatment of this mode analysis appears in the excellent monograph of Chandrasekhar \cite{Chandra1}, to which we refer the reader.

Following Chandrasekhar, we consider axisymmetric perturbations of a central Schwarzschild spacetime, of mass $M$.  Expressing the perturbation in a particular coordinate system, specializing to standard Schwarzschild coordinates on the central spacetime, the linearization of the vacuum Einstein equations \eqref{eq: vacuumEinstein} split naturally into so-called axial and polar perturbations.  Building upon the work of Vishveshwara \cite{Vishveshwara}, we address the issue of linear stability in the axial setting.  We summarize our results as follows:

\begin{theorem}
Axial solutions to the linearized vacuum Einstein equations around a Schwarzschild metric, specified by suitably regular, asymptotically flat initial data,

\begin{enumerate}
\item remain uniformly bounded on the exterior region and
\item decay to a linearized Kerr metric.
\end{enumerate}

\end{theorem}

Our conclusions emerge from a direct examination of the linearized vacuum Einstein equations, yielding the decoupling of connection-level objects.  Interpreted as sections of certain complex line bundles, equivalently as sections of spin bundles, over Schwarzschild spacetime, these quantities satisfy Regge-Wheeler equations of the same sort found in \cite{DHR}.  The vector field methods applied to the scalar wave equation on Schwarzschild carry over to our setting, yielding decay estimates which complete the proof of linear stability.

Until recently the mode analysis upon which we build was the extent of the understanding of linear stability in the Schwarzschild setting.  Although a fuller treatment of the Einstein field equations was long understood to be of primary importance in addressing issues of stability, particularly that of black hole solutions, the analytical framework required in such a treatment has only recently emerged.  Early efforts by Klainerman \cite{Klainerman1, Klainerman2} in the Minkowski spacetime introduced a robust vector field multiplier method, by means of which he obtained uniform boundedness and decay estimates on the classical wave equation, without reliance on explicit solution formulae.  This multiplier method was then extended to more complicated geometric and non-linear settings, culminating in the monumental work of Christodoulou and Klainerman \cite{ChristKlainerman}, later refined in \cite{Bieri, Lindblad, Speck}, establishing the non-linear stability of the Minkowski spacetime.  

As well in the Schwarzschild and Kerr settings, the estimates and techniques from the study of the scalar wave equation, regarded as a ``poor man's'' linearization of the vacuum Einstein equations, are expected to prove an essential ingredient in further progress on linear stability, with developments in linear stability playing a similar role in non-linear stability.  A complete theory is now in place for scalar waves, first appearing in the seminal works of \cite{DR} in the Schwarzschild setting and \cite{DRR} in sub-extremal Kerr, with contributions and refinements also appearing in \cite{KayWald, BlueSterbenz, AnderssonBlue, Tataru1, Tataru2, Smoller, Luk}.  Building upon these results, the recent breakthrough of Dafermos-Holzegel-Rodnianski \cite{DHR} proves linear stability of the Schwarzschild spacetime subject to gravitational perturbations, superseding the earlier modal analysis.

We emphasize that, in contrast with the more general result of Dafermos-Holzegel-Rodnianski, or indeed any of the earlier results on the Minkowski spacetime, our method relies upon the analysis of a connection-level object.  In particular, the authors \cite{DHR} rely upon a multi-level hierarchy of geometric objects, ascending to the level of second derivatives of curvature, with an accordingly high degree of complexity in their analysis. 

\section{Statement of Results}

A linearized solution will consist of the three quantities $(\alpha, \beta, \gamma)$ defined in Section 5.1.  The first two theorems pertain to the quantities $\alpha$ and $\beta$, each satisfying a Regge-Wheeler equation.  Subject to appropriate normalization, by the addition of a suitable linearized Kerr solution, we obtain statements of uniform boundedness and uniform decay for $\alpha$ and $\beta$.

First we define all the relevant energy norms, specified in terms of the red-shift multiplier $N$.  For a section $\xi$, we define the fluxes

\begin{align}
E^{N}_{\xi}(\Sigma_{\tau}) &:=\int_{\Sigma_{\tau}} J^{N}_{a}[\xi]\eta^{a},\\
E^{N}_{\xi}(\tilde{\Sigma}_{\tau}) &:= \int_{\tilde{\Sigma}_{\tau}} J^{N}_{a}[\xi]\eta^{a},
\end{align}

\noindent
in terms of the hypersurfaces $\Sigma_{\tau}$ and $\tilde{\Sigma}_{\tau}$ of Sections 7.3 and 7.7, and the initial energies

\begin{align}\label{energies}
E_0[\xi] &:= \sum_{(m)\leq 2}\int_{\{t = 0 \}} J^{N}_{a}[\Omega^{(m)}\xi]\eta^{a},\\
E_1[\xi] &:= \sum_{(m)\leq 3} \int_{\{t = 0 \}} (1+r_*^2)J^{N}_{a}[\Omega^{(m)}\xi]\eta^{a},\\
E_2[\xi] &:= \sum_{(m)\leq 6} \int_{\{t = 0 \}} (1+r_*^2)J^{N}_{a}[\Omega^{(m)}\xi]\eta^{a},
\end{align}

\noindent
expressed in terms of the (weighted) flux of $N$ through $\{ t = 0\}$ and the angular Killing fields $\Omega_{i}$.

\begin{theorem}
Let $\alpha$ be a solution of the Regge-Wheeler equation (\ref{eq: RW1}), smooth and compactly supported on the time-slice $\{ t = 0 \}$.  Then $\alpha$ satisfies the following estimates:

\begin{enumerate}
\item the uniform energy bound

\begin{equation}
E^{N}_{\alpha}(\Sigma_{\tau}) \leq C E^{N}_{\alpha}(\Sigma_{0}),
\end{equation}

\item the uniform pointwise bound

\begin{equation}
|\alpha|_{C^{0}(\mathcal{D})} \leq C E_0[\alpha],
\end{equation}

\noindent
where $\mathcal{D}= J^{+}(\Sigma_{0})\cap{J^{-}(\mathcal{I}^{+})}$,

\item the uniform energy decay

\begin{equation}
E^{N}_{\alpha}(\tilde{\Sigma}_{\tau}) \leq C E_1[\alpha] \tau^{-2},
\end{equation}

\item the uniform pointwise decay

\begin{equation}
\sup_{\tilde{\Sigma}_{\tau}} |\alpha| \leq C\sqrt{E_2[\alpha]}\tau^{-1}.
\end{equation}

\end{enumerate}

\end{theorem}

\begin{theorem}
Let $\beta$ be defined as in (\ref{betaDef}), solving the Regge-Wheeler equation (\ref{eq: RW2}).  Assuming $\beta$ is smooth and compactly supported on the time-slice $\{t = 0 \}$, then $\beta$ satisfies the following:

\begin{enumerate}

\item with the addition of a suitable linearized Kerr solution, the lowest spherical mode $\beta_1$ can be made to vanish,

\item with $\beta$ supported away from the lowest mode, $\beta = \beta_{\ell > 1},$ we have

\begin{enumerate}
\item the uniform energy bound

\begin{equation}
E^{N}_{\beta}(\Sigma_{\tau}) \leq C E^{N}_{\beta}(\Sigma_{0}),
\end{equation}

\item the uniform pointwise bound

\begin{equation}
|\beta|_{C^{0}(\mathcal{D})} \leq C E_0[\beta],
\end{equation}

\noindent
where $\mathcal{D}= J^{+}(\Sigma_{0})\cap{J^{-}(\mathcal{I}^{+})}$,

\item the uniform energy decay

\begin{equation}
E^{N}_{\beta}(\tilde{\Sigma}_{\tau}) \leq C E_1[\beta] \tau^{-2},
\end{equation}

\item the uniform pointwise decay

\begin{equation}
\sup_{\tilde{\Sigma}_{\tau}} |\beta| \leq C\sqrt{E_2[\beta]}\tau^{-1}.
\end{equation}

\end{enumerate}
\end{enumerate}
\end{theorem}

Using Theorem 2.2, we can normalize and assume $\beta = \beta_{\ell > 1}$, supported away from the lowest spherical mode.  With decay estimates on $\alpha$ and $\beta$, we derive estimates on $\gamma$, and hence decay of the normalized linearized solution $(\alpha, \beta, \gamma)$.  The decay of these three quantities is the content of our main theorem.

\begin{theorem}
Suppose $(\alpha, \beta, \gamma)$ is an axial solution of the linearized Einstein vacuum equations about the Schwarzschild spacetime.  Assume, moreover, that each of the three quantities is smooth and compactly supported on $\{ t = 0 \}$.  Then the following hold:

\begin{enumerate}

\item with the addition of a suitable linearized Kerr solution, the lowest spherical mode $\beta_1$ can be made to vanish, yielding the normalization $\beta = \beta_{\ell > 1}$,

\item each of the three quantities in the normalized solution $(\alpha, \beta, \gamma)$ decay, both in energy and pointwise, through the foliation $\tilde{\Sigma}_{\tau}.$

\end{enumerate}
\end{theorem}

We remark that all the above theorems hold with rougher initial data, with regularity corresponding to the completion of the weighted Sobolev norms specified by the energies (\ref{energies}) and with appropriate decay at infinity.

\section{Axisymmetric Spacetimes}
\subsection{Axisymmetric Coordinates}
Following Chandrasekhar \cite{Chandra1}, our family of axisymmetric spacetimes $(M_{\epsilon}, g_{\epsilon})$ admits the metric expression

\begin{equation}
g_{\epsilon} = -e^{2\nu}(dx^0)^2 + e^{2\mu_2}(dx^2)^2 + e^{2\mu_3}(dx^3)^2 + e^{2\psi}(d\phi -\omega dx^0 - q_2 dx^2 - q_3 dx^3)^2,
\end{equation}

\noindent
where $\nu, \mu_2, \mu_3, \psi, \omega, q_2, q_3$ are functions only of $x^0, x^2, x^3$.  

We further distinguish between the axial perturbations and polar perturbations, decoupling in our study of linearized gravity.  Axial perturbations feature non-vanishing values of $\omega, q_2, q_3$, with $\nu, \mu_2, \mu_3, \psi$ constant in the parameter $\epsilon$ (and hence determined on the Schwarzschild center).  Polar perturbations are just the opposite: $\omega=q_2=q_3 =0$, and the resultant family of diagonal metrics has variation in $\nu, \mu_2, \mu_3, \psi$.  

For further details on the coordinate condition above, see Appendix A.

\subsection{The Axial One-Form $\boldsymbol{\zeta}$}
Key to our study of the axial perturbations are the three quantities $\omega, q_2, q_3$, all appearing in the one-form 

\begin{equation}
\zeta := \omega dx^0 + q_2 dx^2 + q_3 dx^3
\end{equation}

\noindent
within the metric.  Taking the exterior derivative, we have

\begin{equation}
d\zeta = Q_{02} dx^0\wedge dx^2 + Q_{03} dx^0 \wedge dx^3 + Q_{23} dx^2\wedge dx^3,
\end{equation}

\noindent
where 

\begin{align}
Q_{02} := \omega_{2} - q_{2,0},\\ 
Q_{03} := \omega_{3} - q_{3,0},\\
Q_{23} := q_{2,3} - q_{3,2}.
\end{align}

\section{The Schwarzschild Spacetime}
\subsection{The Schwarzschild Metric}
We work primarily within the Schwarzschild exterior, covered by a coordinate patch $(t,r,\theta,\phi)$ with $t\in{\R}, r > 2M, (\theta,\phi)\in{S^2}$.  In these standard Schwarzschild coordinates, the metric has the form

\begin{equation}
 g = -\left(1-\frac{2M}{r}\right)dt^2 + \left(1-\frac{2M}{r}\right)^{-1}dr^2+ r^2d\sigma,
 \end{equation}

\noindent
where $d\sigma = d\theta^2 + \sin^2\theta d\phi^2$ is the round metric on the unit sphere.  Often we use the notation

\begin{align}
\mu & := \frac{2M}{r},\\
\Delta & :=  r^2 - 2Mr.
\end{align}

The form of the metric above indicates that, for an axial perturbation about a central Schwarzschild metric, we have

\begin{align}
\nu = &\frac{1}{2}\log(1-\mu),\\
\mu_2 = &\frac{-1}{2}\log(1-\mu),\\
\mu_3 = &\log r,\\
\psi = & \log (r\sin\theta).
\end{align}

In the course of our analysis, various other coordinate systems will prove useful; we enumerate them below.

First, we have the Regge-Wheeler coordinates, with tortoise coordinate $r_{*}$ normalized as $r_{*} = r + 2M\ln(r-2M) - 3M - 2M\ln(M)$, such that the metric

\begin{equation}
g = -(1-\mu)dt^2 + (1-\mu)dr_{*}^2 + r^2d\sigma
\end{equation}

\noindent
is defined on $t\in{\R}, r_{*}\in{\R}, (\theta, \phi) \in {S^2}$, with $r_{*} = 0$ on the photon sphere $r = 3M$. 

A variant of the above takes $t_{*} = t + 2M\ln(r-2M)$, with 

\begin{equation}
g = -(1-\mu)dt_{*}^2 + 2\mu dt_{*}dr + (1+\mu)dr^2 + r^2d\sigma,
\end{equation}

\noindent
now defined for $(t_{*}, r, \theta, \phi)$ coordinates satisfying $t_{*}\in{\R}, r > 0, (\theta,\phi)\in{S^2}.$

Finally, we make use of the Eddington-Finkelstein double null coordinates, expressed in terms of the Regge-Wheeler coordinates by $ u = \frac{1}{2}(t - r_{*})$ and  $v = \frac{1}{2}(t + r_{*})$, such that

\begin{equation}
g = -4(1-\mu)dudv + r^2d\sigma,
\end{equation}

\noindent
defined for $u\in{\R}, v\in{\R}, (\theta, \phi)\in{S^2}$.

Note that the above coordinate systems cover only a portion of the maximally extended Schwarzschild spacetime.  In particular, to work on the event horizon, corresponding to the coordinate singularity $r = 2M$ in the standard Schwarzschild coordinates, we must employ Kruskal coordinates \cite{Kruskal}.

The Schwarzschild spacetime has a number of Killing fields.  In particular, we have the static Killing field, denoted $T$, and the rotational Killing fields, denoted $\Omega_{i}$, with $i = 1,2,3$.  For convenience in what follows, we collect the rotation Killing fields in the set $\Omega := \{ \Omega_{i} | i = 1,2,3\}$.

For more information on the Schwarzschild spacetime, we direct the reader to the comprehensive references \cite{Wald, Chandra1}.

\subsection{Complex Line Bundles on Schwarzschild Spacetime}
Let $\xi : S^2 \rightarrow \mathbb{C}$ be the stereographic projection.  Expressed in spherical coordinates $(\theta, \phi)$, we have

$$ \xi = \tan \frac{\theta}{2} \cos \phi + i \tan \frac{\theta}{2} \sin \phi. $$

Denote the holomorphic cotangent bundle of $S^2$ by 

\begin{equation}
\mathcal{E}(-2) := \Lambda^{1,0}(S^2).
\end{equation}  

As a line bundle, $\mathcal{E}(-2)$ is spanned by a single section, which we take to be 

\begin{equation}
\psi_{-1} = \frac{1}{\sqrt{2}}(\frac{d \theta}{\sin \theta} + i d\phi).
\end{equation}

Note that the standard Hermitian metric on the sphere gives a cotangent-tangent isomorphism via $d\xi \mapsto g^{\xi \bar{\xi}}\frac{\partial}{\partial \bar{\xi}}$, mapping $\Lambda^{1,0}$ to the anti-holomorphic tangent bundle $T^{0,1}$, which we denote

\begin{equation}
\bar{\mathcal{E}}(2) := T^{0,1}(S^2).
\end{equation}

Using the two, as well as their duals, we can form further complex line bundles by taking repeated tensor products, defined by

\begin{equation}
\mathcal{E}(-2s) := (\mathcal{E}(-2))^{s}
\end{equation}

\noindent
and 

\begin{equation}
\bar{\mathcal{E}}(2s) := (\bar{\mathcal{E}}(2))^{s}.
\end{equation}
The $\partial$ and $\bar{\partial}$ operators map between different line bundle by the following identification:

\begin{equation}
\partial:\ \mathcal{E}(-2s)\longrightarrow\mathcal{E}(-2s)\otimes \Lambda^{1,0}\cong \mathcal{E}(-2s-2)
\end{equation}

\noindent
and 

\begin{equation}
\bar{\partial}:\ \mathcal{E}(-2s)\longrightarrow\mathcal{E}(-2s)\otimes \Lambda^{0,1}\cong \mathcal{E}(-2s+2).
\end{equation}

Later in the analysis, we will be concerned a certain $L^2$ decomposition on the $\mathcal{E}(-2s)$ bundles, into so-called ``spin-weighted spherical harmonics" $Y_{s\ell m}(\theta, \phi)$.  We discuss the details of this decomposition in Appendix \ref{spinHarmonics}.

We extend the $\mathcal{E}$ bundles above to complex line bundles over Schwarzschild spacetime as follows.  Let $\mathcal{L}(-1)$ be the sub-bundle of $\Lambda^{1}_{\mathbb{C}}(Sch)$, for which each fiber is generated by $\frac{d \theta}{\sin \theta} + i d\phi$.   As a line bundle, $\mathcal{L}(-1)$ is spanned by a single section, which we take to be 

\begin{equation}
\Psi_{-1} := \frac{r}{\sqrt{2}}(\frac{d \theta}{\sin \theta} + i d\phi).
\end{equation}

Note that $\mathcal{L}(-1)$ comes equipped with a standard Hermitian structure, as in the round sphere setting.  We regard the Hermitian connection $\slashed{\nabla}$ on $\mathcal{L}(-1)$ as the orthogonal projection of the standard covariant derivative on Schwarzschild onto the subspace spanned by $\Psi_{-1}$. Just as before, we can form further complex line bundles by repeated tensor products. In particular, we shall concern ourselves with the bundles $\mathcal{L}(-1)$ and $\mathcal{L}(-2)$. 

As we are working over a line bundle, the connection is easily characterized.  Namely, for a section $f\Psi_{-1}$ of $\mathcal{L}(-1)$,

\begin{equation}
\slashed{\nabla}_{a}(f\Psi_{-1}) = (\frac{\partial f}{\partial x^{a}} + A_{a}f)\Psi_{-1}, 
\end{equation}

\noindent
with connection one-form

\begin{equation}
A_r = A_t = 0, A_{\theta} = -\cot \theta, A_{\phi} = - i \cos \theta,
\end{equation}

\noindent
extending easily to a computation on $\mathcal{L}(-2)$.

We record the d'Alembertian operators of $\mathcal{L}(-1)$ and $\mathcal{L}(-2)$. For an axial symmetric function $f$, we have

\begin{equation}
\slashed{\Box}_{\mathcal{L}(-1)} (f\Psi_{-1}) = \left[\frac{-r^2}{\Delta}f_{tt} + \frac{\Delta}{r^2}f_{rr} + \frac{\Delta_{r}}{r^2}f_{r} + 
\frac{1}{r^2}(f_{\theta\theta} - \cot\theta f_{\theta} + f)\right]\Psi_{-1},
\end{equation}

\noindent
and

\begin{equation}
\slashed{\Box}_{\mathcal{L}(-2)} (f \Psi_{-1}^2) = \left[\frac{-r^2}{\Delta}f_{tt} + \frac{\Delta}{r^2}f_{rr} + \frac{\Delta_{r}}{r^2}f_{r} + 
\frac{1}{r^2}(f_{\theta\theta} - 3\cot \theta f_{\theta} + 2f) \right]\Psi_{-1}^2,
\end{equation}

\noindent
where we remind the reader $\Delta := r^2 - 2Mr$. Furthermore, the operators $\partial$ and $\bar{\partial}$ can be defined on $\mathcal{L}(-s)$ by computing on each spheres separately. Again let $f$ be an axial symmetric function, we record

\begin{equation}
\partial \left( f\Psi^s_{-1} \right)=\frac{1}{\sqrt{2}r}\sin^{2s+1}\theta\left(\sin^{-2s}\theta f\right)_\theta \Psi^{s+1}_{-1}
\end{equation}

\noindent
and

\begin{equation}
\bar{\partial} \left( f\Psi^{s}_{-1} \right)=\frac{1}{\sqrt{2}r}\sin^{-1}\theta f_\theta \Psi^{s-1}_{-1}.
\end{equation}









\section{Linearized Gravity in the Axial Setting}
\subsection{The Linearized Einstein Equations}
Linearization of the vacuum Einstein equations \eqref{eq: vacuumEinstein} yields just three equations relating the connection terms $Q_{02}, Q_{03},$ and $Q_{23}$, corresponding to the cross-terms with the axial parameter $\phi$.  Specifically, we have

\begin{align}
R_{t\phi}^{(1)} = 0 \rightarrow (r^4\sin^3\theta Q^{(1)}_{02})_{r} + (\frac{r^4}{\Delta}\sin^3\theta Q^{(1)}_{03})_{\theta} = 0,\\
R_{r\phi}^{(1)} = 0 \rightarrow (\Delta \sin^3\theta Q^{(1)}_{23})_{\theta} + (r^4\sin^3\theta Q^{(1)}_{02})_{t} = 0,\\
R_{\theta\phi}^{(1)} = 0 \rightarrow (\Delta \sin^3\theta Q^{(1)}_{23})_{r} - (\frac{r^4}{\Delta}\sin^3\theta Q^{(1)}_{03})_{t} = 0.
\end{align}

Furthermore, as $d^2\zeta = 0$, we have the closed condition

\begin{equation}
d^2\zeta^{(1)} = 0 \rightarrow Q^{(1)}_{02,\theta} - Q^{(1)}_{03,r} + Q^{(1)}_{23,t} = 0.
\end{equation}

For each of the linearized connection terms, we study an associated section in a holomorphic line bundle $\mathcal{L}(-s)$.  Namely, defining

\begin{align}
 \alpha &:= \frac{\Delta}{r^2}\sin^3\theta Q_{23}^{(1)} (\Psi_{-1})^2,\\
 \beta &:= r^2\sin^2\theta Q_{02}^{(1)}\Psi_{-1},\label{betaDef}\\
 \gamma &:= \sin^3\theta Q_{03}^{(1)}(\Psi_{-1})^2,
\end{align}

\noindent
as sections of $\mathcal{L}(-2), \mathcal{L}(-1)$, and $\mathcal{L}(-2)$, respectively, we rewrite the four governing equations above as 

\begin{align}
R_{t\phi}^{(1)} = 0 &\rightarrow (r^2\beta)_{r} + \frac{r^4}{\Delta}\sqrt{2}r\bar{\partial}\gamma = 0,\label{Rtphi}\\
R_{r\phi}^{(1)} = 0 &\rightarrow \sqrt{2}r\bar{\partial}\alpha + \beta_{t},\label{Rrphi}\\
R_{\theta\phi}^{(1)} = 0 &\rightarrow (r^2\alpha)_{r} -\frac{r^4}{\Delta}\gamma_{t} = 0,\label{Rthetaphi}\\
d^2\beta^{(1)} = 0 &\rightarrow \frac{1}{r^2}\sqrt{2}r\partial\beta - \gamma_{r} + \frac{r^2}{\Delta}\alpha_{t} = 0.\label{betaClosed}
\end{align}

\subsection{Description of the Kernel}

The Kerr solution splits at the linear level into two parts: a polar part giving change of mass, and an axial part giving change in angular momentum.  For this axial portion, direct examination of the metric shows 

\begin{align}
\begin{split}
q_2^{(1)} &= q_3^{(1)} = 0,\\
\omega^{(1)} &= \frac{2Ma^{(1)}}{r^3},
\end{split}
\end{align}

\noindent
such that

\begin{align}
\begin{split}
Q_{03}^{(1)} &= Q_{23}^{(1)} = 0,\\
Q_{02}^{(1)} &= \frac{-6Ma^{(1)}}{r^4},
\end{split}
\end{align}

\noindent
and

\begin{align}\label{KerrData}
\begin{split}
\alpha &= \gamma = 0,\\
\beta &= \frac{-6Ma^{(1)}}{r^2}\sin^2\theta \Psi_{-1}.
\end{split}
\end{align}

Note that our condition of a fixed axis of symmetry for the perturbation simplifies our description of linearized Kerr considerably, in comparison with \cite{DHR}.

The pure gauge amounts to taking a change of coordinates 

\begin{align*}
\tilde{t} &= t + \epsilon f_0(t,r,\theta), \\
\tilde{\phi} &= \phi + \epsilon f_1(t,r,\theta),\\
\tilde{r} &= r + \epsilon f_2(t,r,\theta),\\
\tilde{\theta} &= \theta + \epsilon f_3(t,r,\theta).
\end{align*}

\noindent
The assumption of axisymmetry,  the form of the axial metric at the linear level, and the choice of a Schwarzschild base metric, taken together, yield $f_0 = f_2 = f_3 = 0$ straightaway.  This being the case, the gauge change corresponds to a function $f_1(t, r, \theta)$ such that $df_1 = \zeta^{(1)}$.  Exactness of $\zeta^{(1)}$ implies vanishing of the connection terms, and pure gauge solutions are seen to be trivial.
 
\section{Decoupling of $\alpha$}
The four governing equations in the previous section allow for separation of $\alpha$, described below.  Applying the operator $\sqrt{2}\partial$ to (\ref{Rrphi}) and differentiating (\ref{Rthetaphi}) in $r$ and (\ref{betaClosed}) in $t$, we find

\begin{gather*}
2\partial\bar\partial \alpha + \frac{\sqrt{2}}{r}\partial \beta_{t} = 0,\\
\left(\frac{\Delta}{r^4}(r^2\alpha)_{r}\right)_{r} - \gamma_{t,r} = 0,\\
-\frac{\sqrt{2}}{r}\partial \beta_{t} + \gamma_{t,r} - \frac{r^2}{\Delta}\alpha_{tt} = 0.
\end{gather*}

Adding the three equations together, we isolate $\alpha$:

$$ \frac{-r^2}{\Delta} \alpha_{tt} + \left(\frac{\Delta}{r^4}(r^2\alpha)_{r}\right)_{r} + 2\partial\bar\partial\alpha = 0.$$

Expanding and noting the relation between the Laplace-Beltrami and Laplace-de Rham operators

$$ \Delta_{\mathcal{E}(-2s)} = \partial\bar\partial + s,$$

\noindent
we find

\begin{equation}
\frac{-r^2}{\Delta}\alpha_{tt} + \frac{\Delta}{r^2}\alpha_{rr} + \frac{\Delta_{r}}{r^2}\alpha_{r} + \Delta_{\mathcal{E}(-4)}\alpha = V\alpha,
\end{equation}

\noindent
where $V = \frac{4}{r^2}(1-\mu)$.  Rewritten in this form, $\alpha$ is seen to be a solution of a Regge-Wheeler equation

\begin{equation}\label{eq: RW1}
\slashed{\Box}_{\mathcal{L}(-2)} \alpha = V\alpha
\end{equation}

\noindent
with potential $V = \frac{4}{r^2}(1-\mu) = \frac{4}{r^2}\left(1-\frac{2M}{r}\right).$

\section{Analysis of $\alpha$}

\subsection{Stress-Energy Formalism}
Our analysis relies upon the vector field multiplier method, outlined in Appendix \ref{stressEnergy}.

Associated with our wave equation is the stress-energy tensor:

\begin{equation}
T_{ab}[\alpha] := \frac{1}{2}(\slashed{\nabla}_a \alpha \slashed{\nabla}_b \bar{\alpha} + \slashed{\nabla}_a \bar{\alpha} \slashed{\nabla}_b \alpha) - \frac{1}{2}g_{ab}(\slashed{\nabla}^{c} \alpha \slashed{\nabla}_{c}\bar{\alpha} + V|\alpha|^2).
\end{equation}

Applying a vector field multiplier $C^{b}$, we define the energy current

\begin{equation}
J^{C}_{a}[\alpha] := T_{ab}[\alpha]C^{b}
\end{equation}  

\noindent
and the density

\begin{equation}
K^{C}[\alpha] := \nabla^{a}J^{C}_{a}[\alpha] = \nabla^{a}(T_{ab}[\alpha]C^{b}).
\end{equation}

As well, we will have occasion to use the weighted energy current

\begin{equation}
J^{C,\omega^{C}}_{a} := J^{C}_{a} + \frac{1}{4}\omega^{C}\nabla_{a}|\alpha|^2 -\frac{1}{4}\nabla_{a}\omega^{C}|\alpha|^2, 
\end{equation}

\noindent
with weighted density

\begin{equation}
K^{C,\omega^{C}} := K^{C} + \frac{1}{4}\omega^{C}\Box |\alpha|^2 -\frac{1}{4} \Box \omega^{C} |\alpha|^2,
\end{equation}

\noindent
for a suitable scalar weight function $\omega^{C}$.

We typically suppress the $\alpha$ dependence in the notation above when there is no chance of confusion.

The current $J^{C}_{b}$ and density $K^{C}$ serve as a convenient notation to express the spacetime Stokes' theorem

\begin{equation}
\int_{\mathcal{\partial D}} J^{C}_{a}\eta^{a} = \int_{\mathcal{D}} K^{C},
\end{equation} 

\noindent
the primary tool in proving the estimates that follow.

Finally, we observe that, in contrast with the wave equation, the stress-energy tensor $T_{ab}$ has non-trivial divergence

\begin{equation}
\nabla^{a}T_{ab} = \frac{-1}{2}\nabla_{b}V|\alpha|^2 + \frac{1}{2}\slashed{\nabla}^{a}\alpha[\slashed{\nabla}_{a},\slashed{\nabla}_{b}]\bar{\alpha} + \frac{1}{2}\slashed{\nabla}^{a}\bar{\alpha}[\slashed{\nabla}_{a},\slashed{\nabla}_{b}]\alpha,
\end{equation}

\noindent
where we note crucially that the commutator $[\slashed{\nabla}_{a},\slashed{\nabla}_{b}]$ vanishes when contracted with a multiplier invariant under the members of $\Omega$.  In particular, all such multipliers considered in the analysis below have this property.

The notation for our covariant derivative $\slashed{\nabla}$, usually reserved for the angular derivative, presents some chance of confusion.  In particular, we are careful to distinguish the spacetime gradient 

$$ \slashed{\nabla}^{c}\alpha \slashed{\nabla}_{c}\bar\alpha,$$

\noindent
contracting over all spacetime indices, and the angular gradient

\begin{equation}
|\tilde{\slashed{\nabla}}\alpha|^2 = r^{-2}\left( |\slashed{\nabla}_\theta\alpha|^2+\sin^{-2}\theta|\slashed{\nabla}_\phi\alpha|^2 \right).
\end{equation}

As well, we remark that, when integrating over a spacelike hypersurface $\Sigma$, we often use $\nabla_{\Sigma}$ to refer to the associated gradient.

\subsection{Degenerate Energy}
Applying the Killing multiplier $T = \partial_{t}$ over the spacetime region bounded by time slices $\{ t = \tau'\}$ and $\{ t = \tau \}$, we obtain a conservation 
law

$$ \int_{\{t = \tau \}} J^{T}_{a}\eta^{a} = \int_{\{t = \tau' \}} J^{T}_{a}\eta^{a},$$

\noindent
with $\eta^{a}$ being the appropriate unit normal.

Note that the density $K^{T}$ vanishes in this case as a consequence of $V$ being radial, such that $(\nabla^{a}T_{ab})T^{b}\sim{T(V)|\alpha|^2}$ vanishes, and $T$ being Killing, such that $\pi^{T}$ vanishes.

Defining the $T$-energy by

\begin{equation}
E^{T}_{\alpha}(\tau) := \int_{\{t = \tau\}} J^{T}_{a}\eta^{a},
\end{equation}

\noindent
our conservation law is nothing more that the statement

\begin{equation}\label{eq: ConservationLaw}
E_{\alpha}^{T}(\tau) = E_{\alpha}^{T}(\tau')
\end{equation}

\noindent
for all $\tau$ and $\tau'$. 

The trouble with the energy above is that it degenerates near the event horizon, where the multiplier $T$ becomes null.  To remedy this, we rely upon the red-shift multiplier of Dafermos and Rodnianski \cite{DR}, introduced in the next subsection.
 
 \subsection{Non-degenerate Energy}
 
The degeneracy in the $T$-energy can be rectified by recourse to the red-shift multiplier $N$, introduced by Dafermos and Rodnianski in \cite{DR}.  We recount the essential details here, with a fuller treatment in Append \ref{redShift}.  The multiplier $N$ is strictly future-directed time-like on the exterior, up to and on the horizon, and so yields a positive-definite energy flux through space-like hypersurfaces, in contrast with the Killing multiplier $T$.  Distinguishing it from other such choices of future-directed time-like multipliers, however, is the coercive control  

$$ K^{N} \geq c J^{N}_{a}N^{a}, $$

\noindent
for some small $c > 0$, near the event horizon.

Comparing our situation to that of the wave equation, dealt with in the appendix, we note that the pointwise estimate

$$ T_{ab}\nabla^{a}N^{b} \geq{ c J^{N}_{a}N^{a}},$$

\noindent
still holds in our setting, as a consequence of non-trivial surface gravity $\kappa$ and tensorial algebra on $T_{ab}$.  In addition, the divergence term

$$ (\nabla^{a}T_{ab})N^{b} = \frac{-1}{2}N(V)|\alpha|^2 \geq 0,$$

\noindent
is non-negative, as the potential $V$ increases radially near the event horizon.  Taken together, the two yield the desired estimate

$$ K^{N} \geq {T_{ab}\nabla^{a}N^{b}} \geq{c J^{N}_{a}N^{a}}.$$

Extending to the exterior, we construct a strictly timelike multiplier $N$, satisfying the estimates
 
 \begin{align}\label{Nestimates}
 K^{N} &\geq{c J^{N}_{a}N^{a}},\ &&\textup{for}\ 2M\leq r \leq r_0,\\
 J^{N}_{a}N^{a} &\sim J^{T}_{a}T^{a}, \ &&\textup{for}\ r_0 \leq r \leq R_0,\\
 |K^{N}| &\leq  C|J^{T}_{a}T^{a}|, \ &&\textup{for}\ r_0 \leq r \leq R_0,\\
 N &= T, \ &&\textup{for}\ r\geq R_0,
 \end{align}
 
 \noindent
 for some fixed $2M < r_0 < R_0 < \infty$.

Denoting

\begin{equation}
\Sigma_{\tau} := \{ t_{*} = \tau \},
\end{equation}

\noindent
where we remind the reader that $t_{*} = t + 2M\ln(r-2M)$, we define the non-degenerate $N$-energy
 
\begin{equation}
E^{N}_{\alpha}(\Sigma_{\tau}) :=  \int_{\Sigma_{\tau}} J^{N}_{a}\eta^{a}.
\end{equation}

\subsection{Uniform Boundedness}
 
Having adapted the red-shift multiplier $N$, we obtain a uniform energy bound, spelled out in the theorem below.
 
\begin{theorem}\label{EnergyBoundA}
Suppose $\alpha$ is a solution of (\ref{eq: RW1}), smooth and compactly supported on the space-like hypersurface $\Sigma_0$.   Then, for $\tau \geq 0$, $\alpha$ satisfies the uniform energy estimate
 
\begin{equation}
E^{N}_{\alpha}(\Sigma_{\tau}) \leq CE^{N}_{\alpha}(\Sigma_0).
\end{equation}

\end{theorem}

\begin{proof}
The proof proceeds just as in \cite{DRClay}.  Denote by $\mathcal{R}(\tau', \tau)$ the spacetime region between the hypersurfaces $\Sigma_{\tau'}$ and $\Sigma_{\tau}$.  Concretely, $\mathcal{R}(\tau', \tau) = (J^{+}(\Sigma_{\tau'})\setminus J^{+}(\Sigma_{\tau}))\cap J^{-}(\mathcal{I}^{+})$.

Applying the divergence theorem in the spacetime region $\mathcal{R}(\tau', \tau)$, we find

 \begin{align*}
&\int_{\Sigma_{\tau}\cap\{2M\leq r\leq r_0\}}J^N_{a} \eta_{\tau}^{a}+ \int_{\Sigma_{\tau}\cap{\{r\geq{r_0}\}}}J^{N}_{a}\eta_{\tau}^{a} + \int_{\mathcal{R}(\tau', \tau)\cap{\{2M \leq r \leq r_0\}}}K^N\\
\leq &\int_{\Sigma_{\tau'}\cap\{2M\leq r\leq r_0\}}J^N_{a} \eta_{\tau'}^{a}+\int_{\Sigma_{\tau'}\cap\{r\geq{r_0}\}}J^N_{a} \eta_{\tau'}^{a}+\int_{\mathcal{R}(\tau',\tau)\cap\{r_0\leq r\leq R_0\}}|K^Y|,
\end{align*}

\noindent
as the event horizon term has a good sign.

Utilizing as well monotonicity of the $T$-energy $E_{\alpha}^{T}(\Sigma_{\tau})$ and \eqref{Nestimates}, we deduce the integral inequality

$$ E_{\alpha}^{N}(\Sigma_{\tau}) + c\int_{\tau'}^{\tau} E_{\alpha}^{N}(\Sigma_{s}) ds \leq C(\tau - \tau') + E_{\alpha}^{N}(\Sigma_{\tau'}). $$

This inequality implies the uniform energy bound

$$ E_{\alpha}^{N}(\Sigma_{\tau}) \leq C E_{\alpha}^{N}(\Sigma_0),$$

\noindent
where we emphasize that the $N$-energy is positive-definite. 

\end{proof}

Indeed, this estimate extends to more general, non-compactly supported solutions $\alpha$ to the Regge-Wheeler equation \eqref{eq: RW1}.  As well, the above proves to be the case with any reasonable family of hypersurfaces $\Sigma_{\tau}$, formed by translation of an initial space-like hypersurface $\Sigma_0$, where $\Sigma_0$ passes through $\mathcal{H}^{+}$ away from the bifurcation sphere and where we have some uniform control over $ -g_{ab}T^{a}\eta^{b}$, now with $\eta^{b}$ being the normal to $\Sigma_0$.
 
Using the Killing fields of Schwarzschild and certain Sobolev inequalities, it is possible to strengthen the uniform energy estimate above to a statement of uniform boundedness, as follows.

First, let us note that the Killing fields $T$ and $\Omega_{i}$ commute with $\slashed{\Box}$, and, even more, our potential function $V = \frac{4}{r^2}(1-\frac{2M}{r})$ is purely radial.  Taken together, the two show that, for $\alpha$ a solution, so too are the Lie derivatives $\mathcal{L}_{T}\alpha, \mathcal{L}_{\Omega_{i}}\alpha$, and likewise for higher derivatives.  This being the case, the energy estimate just derived carry over for these higher derivatives.  As a shorthand, we denote

\begin{align}
T\alpha := & \mathcal{L}_{T}\alpha,\\
\Omega_{i}\alpha := & \mathcal{L}_{\Omega_{i}}\alpha,\\
\Omega_{i}\Omega_{j}\alpha := & \mathcal{L}_{\Omega_{i}}\mathcal{L}_{\Omega_{j}}\alpha,
\end{align}

\noindent
and likewise for higher derivatives.

\begin{theorem}\label{PointwiseBoundA}
For $\alpha$ satisfying the hypotheses of Theorem \ref{EnergyBoundA} above, we have the pointwise bound

\begin{equation}
||\alpha||_{C^{0}(\mathcal{D})}  \leq {C\left(E^{N}_{\alpha}(\Sigma_0) + \sum_{i =1}^{3}E^{N}_{\Omega_{i}\alpha}(\Sigma_0) + \sum_{i,j = 1}^{3} E^{N}_{\Omega_{i}\Omega_{j}\alpha}(\Sigma_0)\right)},
\end{equation}

\noindent
with $\mathcal{D} = J^{+}(\Sigma_0)\cap{J^{-}(\mathcal{I}^{+})}$ being the exterior spacetime region foliated by the non-negative $\Sigma_{\tau} = \{ t_{*} = \tau\}$ slices.

\end{theorem}

\begin{proof}
As mentioned earlier, the non-degenerate energy estimate just derived carries over for for higher derivatives in Killing directions.  In particular,

\begin{align}
E^{N}_{\Omega_{i}\alpha}(\Sigma_{\tau})  \leq & C E^{N}_{\Omega_{i}\alpha}(\Sigma_0) \\
E^{N}_{\Omega_{i}\Omega_{j}\alpha}(\Sigma_{\tau}) \leq & C E^{N}_{\Omega_{i}\Omega_{j}\alpha}(\Sigma_0)
\end{align}

With these higher derivative estimates in hand,  the Sobolev embedding theorems, and subsequent recovery of a pointwise estimate, apply as follows.

 Let $\rho$ be a geodesic radial coordinate on $\Sigma_\tau$ and $\rho=1$ on the horizon. Then for any $\omega\in S^2$,
 
\begin{align*}
|\alpha(\rho,\omega)|&\leq \int_\rho^\infty |\slashed{\nabla}_{\rho}\alpha|(\rho',\omega)d\rho'\\
&\leq\left(\int_\rho^\infty |\slashed{\nabla}_{\rho}\alpha|^2(\rho',\omega)(\rho')^2d\rho'\right)^{1/2}\left(\int_\rho^\infty(\rho')^{-2}d\rho'\right)^{1/2}\\
&\leq C\left(\int_\rho^\infty |\slashed{\nabla}_{\rho}\alpha|^2(\rho',\omega)(\rho')^2d\rho'\right)^{1/2}.
\end{align*}

The Sobolev inequality on $S^2$ yields the estimate

\begin{align*}
|\slashed{\nabla}_{\rho}\alpha|^2(\rho',\omega)\leq C\int_{S^2}\left[ |\slashed{\nabla}_{\rho}\alpha|^2+\sum_{i = 1}^{3}|\slashed{\nabla}_{\rho}\Omega_{i}\alpha|^2+\sum_{i,j=1}^{3}|\slashed{\nabla}_{\rho}\Omega_{i}\Omega_{j}\alpha|^2 \right]d\sigma,
\end{align*}

\noindent
where $d\sigma$ is the area form of $S^2$. Inserting into the first, we find

\begin{align*}
|\alpha(\rho,\omega)|^2&\leq C\left( \int_\rho^\infty \int_{S^2} (\rho')^2 d\rho' d\sigma \left[|\nabla_{\Sigma_\tau} \alpha|^2+\sum_{i=1}^{3}|\nabla_{\Sigma_\tau} \Omega_{i}\alpha|^2+\sum_{i,j=1}^{3}|\nabla_{\Sigma_\tau}\Omega_{i}\Omega_{j}\alpha|^2\right] \right)\\
&\leq C\left(E^{N}_{\alpha}(\Sigma_0) + \sum_{i =1}^{3}E^{N}_{\Omega_{i}\alpha}(\Sigma_0) + \sum_{i,j = 1}^{3} E^{N}_{\Omega_{i}\Omega_{j}\alpha}(\Sigma_0)\right).
\end{align*}

It is clear from the proof that, in fact, one has $|\alpha|\leq C/\sqrt{r}$.
\end{proof}

 \subsection{The Morawetz Multiplier $X$}
 
Next we deduce a degenerate form of integrated decay, wherein we bound a spacetime integral of various weighted $L^2$ quantities by initial data. 
 
Let $X = f(r_{*})\partial_{r_{*}}$ be a generic radial multiplier, where we use the Regge-Wheeler coordinate $r_{*}$ specified in the first section.  Further, denote by $g'$ the differential of a function $g$ in $r_{*}$.

We calculate the pieces of the density $K^{X}$ to be

$$T^{ab}\pi^{X}_{ab} = \frac{f'}{1-\mu}|\slashed{\nabla}_{r_{*}}\alpha|^2 + \frac{f}{r}(1-\frac{3M}{r})|\tilde{\slashed{\nabla}}\alpha|^2-\frac{1}{4}(2f' + 4f\frac{1-\mu}{r})\slashed{\nabla}^{c}\alpha \slashed{\nabla}_{c}\bar{\alpha} - \frac{1}{2}(f' + \frac{2f}{r}(1-\frac{M}{r}))V|\alpha|^2,$$

$$(\nabla^{a}T_{ab})X^{b} = \frac{-1}{2}fV'|\alpha|^2 .$$

Combining the previous two, we find  

\begin{multline}
 K^{X} = \frac{f'}{1-\mu}|\slashed{\nabla}_{r_{*}}\alpha|^2 + \frac{f}{r}(1-\frac{3M}{r})|\tilde{\slashed{\nabla}}\alpha|^2\\
- \left[\frac{1}{2}(f' + 2f\frac{1-\mu}{r})\slashed{\nabla}^{c}\alpha\slashed{\nabla}_{c}\bar{\alpha} + \frac{1}{2}(f' + \frac{2f}{r}(1-\frac{M}{r}))V|\alpha|^2+\frac{1}{2}fV'|\alpha|^2\right].
\end{multline}

Note that the identity 

\begin{equation}
\Box{|\alpha|^2} = 2V|\alpha|^2 + 2\slashed\nabla^{c}\alpha\slashed\nabla_{c}\bar{\alpha}
\end{equation}

\noindent
holds in our bundle setting.  Moreover, we can encode integration by parts on the d'Alembertian term by using the weighted energy current

\begin{equation}
J^{X,\omega^{X}}_{a} := J^{X}_{a} + \frac{1}{4}\omega^{X}\nabla_{a}|\alpha|^2 -\frac{1}{4}\nabla_{a}\omega^{X}|\alpha|^2, 
\end{equation}

\noindent
with density

\begin{equation}
K^{X,\omega^{X}} := K^{X} + \frac{1}{4}\omega^{X}\Box |\alpha|^2 -\frac{1}{4} \Box \omega^{X} |\alpha|^2.
\end{equation}

Inserting the weight 

\begin{equation}\label{eq: Xweight}
\omega^{X} := f' + 2f\frac{1-\mu}{r}
\end{equation}

\noindent
and the potential $V = \frac{4}{r^2}(1-\mu)$, we find after some simplification
 
\begin{equation}
K^{X,\omega^{X}} =  \frac{f'}{1-\mu}|\slashed{\nabla}_{r_{*}}\alpha|^2 + \frac{f}{r}(1-\frac{3M}{r})|\tilde{\slashed{\nabla}}\alpha|^2
+ \left[\frac{-Mf}{r^2}V -\frac{1}{2}fV' -\frac{1}{4}\Box \omega^{X} \right]|\alpha|^2.
\end{equation}

Our objective is a choice of $f$ yielding a positive bulk term.  As it turns out, such an estimate becomes clear not when integrating over an entire spacetime region $\mathcal{D}$, but over spheres of symmetry within.  

A cursory examination of the above reveals that such an $f$ ought to have a positive $r_{*}$ derivative, and moreover ought to vanish at the photon sphere $r = 3M$ to odd order.  Even with these conditions satisfied, the search for an $f$ yielding a positive coefficient on the $L^2$ term $|\alpha|^2$ seems hopeless.  What is required is a borrowing from the spherical $L^2$-gradient term, i.e. a Poincar\'{e}-type inequality, as in the argument of Alinhac \cite{Alinhac1} for higher modes of the scalar wave and that of Dafermos-Holzegel-Rodnianski \cite{DHR}.  

Letting

\begin{equation}\label{eq: XmultDef}
f := \Big{(}1+\frac{M}{r}\Big{)}^2\Big{(}1-\frac{3M}{r}\Big{)},
\end{equation}

\noindent
we compute

\begin{equation}
f' = \frac{M}{r^2}(1-\mu)\Big{(}1+\frac{M}{r}\Big{)}\Big{(}1+\frac{9M}{r}\Big{)},
\end{equation}

\noindent
and note that both $f$ and $f'$ are uniformly bounded on $r \geq 2M$, with $f'$ positive away from the event horizon.

We calculate the weighted density $K^{X,\omega^{X}}$ associated to this choice of $f$, using the formula

$$\Box{\omega^{X}} = \Box(f' + 2f\frac{1-\mu}{r}) = \frac{1}{2(1-\mu)}f''' + \frac{2}{r}f'' -\frac{2\mu'}{r(1-\mu)}f' +\frac{1}{(1-\mu)r}\left(\frac{\mu'(1-\mu)}{r}-\mu''\right)f.$$

When the smoke clears, we have coefficients of good (i.e. positive) sign preceding the radial and angular $L^2$ terms.  On the other hand, the coefficient before the base $L^2$ term behaves badly.  Namely, we find 

\begin{multline}
\frac{-Mf}{r^2}V -\frac{1}{2}fV' -\frac{1}{4}\Box \omega^{X} \\ = \frac{-534 M^5 - 244 M^4 r + 304 M^3 r^2 + 118 M^2 r^3 - 105 M r^4 + 16 r^5}{4 r^8},
\end{multline}

\noindent
with negative values in a spatially compact region away from the photon sphere.  However, as $\alpha$ is a section of the bundle $\mathcal{L}(-2)$, we have the Poincar\'{e} inequality

\begin{equation}
\int_{S^2(r)} |\tilde{\slashed{\nabla}}\alpha|^2  \geq{\frac{2}{r^2}\int_{S^2(r)} |\alpha|^2}.
\end{equation}

Using this estimate, we borrow from the angular term to deduce

\begin{multline}
\int_{S^2(r)} \frac{f}{r}(1-\frac{3M}{r})|\tilde{\slashed{\nabla}}\alpha|^2 +
\left[\frac{-Mf}{r^2}V -\frac{1}{2}fV' -\frac{1}{4}\Box \omega^{X} \right]|\alpha|^2 \\
\geq{\int_{S^2(r)} \frac{-534 M^5 - 172 M^4 r + 400 M^3 r^2 + 102 M^2 r^3 - 137 M r^4 + 
24 r^5}{4 r^8}}|\alpha|^2 \geq{c\int_{S^2(r)} \frac{|\alpha|^2}{r^3}},
\end{multline}
 
\noindent
for $c$ a small positive constant.
 
Integrating over the spherically symmetric spacetime region bounded by time slices $\{ t = \tau'\}$ and $\{ t = \tau \}$, with $\tau \geq \tau' \geq 0$, we find
 
\begin{multline}
\int_{\tau' \leq{t}\leq{\tau}} K^{X,\omega^{X}} = \int_{\tau'}^{\tau}\int_{2M}^{\infty}\int_{S^2}K^{X,\omega^{X}} r^2d\sigma dr dt\\
\geq{c\int_{\tau'}^{\tau}\int_{2M}^{\infty}\int_{S^2}\left[\frac{1}{r^2}|\slashed{\nabla}_{r*}\alpha|^2 + \frac{1}{r^3}|\alpha|^2\right]r^2d\sigma dr dt.}
\end{multline}
 
Indeed, it is possible to borrow somewhat less of the angular term, yielding the improvement
 
 \begin{multline}
  \int_{\tau' \leq{t}\leq{\tau}} K^{X,\omega^{X}} = \int_{\tau'}^{\tau}\int_{2M}^{\infty}\int_{S^2}K^{X,\omega^{X}} r^2d\sigma dr dt\\
 \geq{c\int_{\tau'}^{\tau}\int_{2M}^{\infty}\int_{S^2}\left[\frac{1}{r^2}|\slashed{\nabla}_{r*}\alpha|^2 + \frac{1}{r^3}|\alpha|^2 + \frac{(r-3M)^2}{r^3}|\tilde{\slashed{\nabla}}\alpha|^2 \right]r^2d\sigma dr dt.}
 \end{multline}
 
Even more, we can estimate the boundary terms $\int_{\{ t = \tau \}} J^{X,\omega^{X}}_{a}\eta^{a}$ by our degenerate $T$-energy $E^{T}_{\alpha}(\tau)$.  We remind the reader,
 
 $$ J^{X,\omega^{X}}_{a} := J^{X}_{a} + \frac{1}{4}\omega^{X}\nabla_{a}|\alpha|^2 -\frac{1}{4}\nabla_{a}\omega^{X}|\alpha|^2, $$
 
 \noindent
 where $\omega^{X}:= f' + 2f\frac{1-\mu}{r}$ for our choice $f:= \Big{(}1+\frac{M}{r}\Big{)}^2\Big{(}1-\frac{3M}{r}\Big{)}.$
 
 It suffices to estimate each of the pieces of $J^{X,\omega^{X}}_{a}$.  Noting that $\eta^{a} = (1-\mu)^{-1/2}\partial_{t}$, we compute
 
 $$J^{X}_{a}\eta^{a} = f(1+\mu)^{-1/2}\slashed{\nabla}_{t}\alpha \cdot \slashed{\nabla}_{r_{*}}\alpha = f(1-\mu)^{1/2}\slashed{\nabla}_{t}\alpha \cdot \slashed{\nabla}_{r}\alpha,$$
 
 \noindent
 such that
 
 $$ \Big{|}\int_{\{t = \tau\}} J^{X}_{a}\eta^{a} \Big{|} \leq B\int_{\{t = \tau \}} \left[|\slashed{\nabla}_{t}\alpha|^2 + (1-\mu)|\slashed{\nabla}_{r}\alpha|^2\right] \leq CE^{T}_{\alpha}(\tau),$$
 
 \noindent
where we have used uniform boundedness $|f| \leq B$ of $f$.  

Next, direct computation gives
 
 $$ \frac{1}{4}\omega^{X}\nabla_{a}|\alpha|^2 \eta^{a} = \frac{(r+M)(15M^2 +5Mr - 2r^2)}{r^4}(1-\mu)^{1/2}\slashed{\nabla}_{t}\alpha \cdot \alpha.$$
 
 The term $\frac{(r+M)(15M^2 +5Mr - 2r^2)}{r^4}$ is uniformly bounded in the exterior region; applying Young's inequality, we find
 
 $$ \Big{|}\int_{\{ t = \tau \}} \frac{1}{4}\omega^{X}\nabla_{a}|\alpha|^2 \eta^{a}\Big{|} \leq B \int_{\{ t = \tau \}} \left[|\slashed{\nabla}_{t}\alpha|^2 + (1-\mu)|\alpha|^2\right] \leq C E^{T}_{\alpha}(\tau).$$
 
 The third term proves to be trivial.  Namely, as $\omega^{X}$ is purely radial, 
 
 $$ \frac{1}{4}|\alpha|^2\nabla_{a}\omega^{X}\eta^{a} = 0.$$
  
 Putting it all together, we bound the weighted energy flux by the initial $T$-energy; that is,
 
 \begin{equation}
 \Big{|}\int_{\{t = \tau\}} J^{X,\omega^{X}}_{a}\eta^{a}\Big{|} \leq C E^{T}_{\alpha}(\tau) \leq CE^{T}_{\alpha}(0),
 \end{equation}
 
 \noindent
 utilizing as well conservation of the $T$-energy \eqref{eq: ConservationLaw}.
 
Applying Stokes' theorem, we summarize this section with the following spacetime estimate
 
\begin{lemma}\label{DegIntegratedDecayA}
Suppose $\alpha$ is a solution of (\ref{eq: RW1}), smooth and compactly supported on the hypersurface $\{ t = 0 \}$.  Applying the multiplier $X=f(r)\partial_{r_{*}}$ as defined above (\ref{eq: XmultDef}), with weight $\omega^{X}$ (\ref{eq: Xweight}), $\alpha$ is seen to satisfy the degenerate integrated decay estimate
  
\begin{multline}
\int_{\tau'}^{\tau}\int_{2M}^{\infty}\int_{S^2}\left[\frac{1}{r^2}|\slashed{\nabla}_{r*}\alpha|^2 + \frac{1}{r^3}|\alpha|^2 + \frac{(r-3M)^2}{r^3}|\tilde{\slashed{\nabla}}\alpha|^2\right]r^2d\sigma dr dt\\
\leq{C\int_{\{\tau' \leq t \leq \tau\}} K^{X,\omega^{X}}} \leq {C\int_{\{ t = \tau'\}} J^{T}_{a}\eta^{a}} = CE^{T}_{\alpha}(\tau'),\end{multline}
 
\noindent
with the spacetime term degenerating at the event horizon $r = 2M$, at the photon sphere $r = 3M$, and at infinity.
 
\end{lemma}
 
 \subsection{The Multiplier $Z$}

The following treatment adheres to the original approach of Dafermos-Rodnianski \cite{DR}, itself adapted from Morawetz \cite{Morawetz}.

We apply the analog of the Minkowskian conformal Killing field $Z$, defined to be 

\begin{equation}
Z:=u^2\frac{\partial}{\partial u}+v^2\frac{\partial}{\partial v} = \frac{1}{2}(t^2+r_{*}^2)\partial_{t} + tr_{*}\partial_{r_{*}},
\end{equation}

\noindent
in either the Eddington-Finkelstein or Regge-Wheeler coordinates.  As the Regge-Wheeler coordinate $r_{*}$ features in the coefficients, we remind the reader of the normalization $r_{*} = r + 2M\ln(r-2M) - 3M - 2M\ln(M)$.

The density $K^{Z}$ is the sum of

\begin{align*}
T_{ab}\nabla^{a} Z^{b} &= t|\tilde{\slashed \nabla}\alpha|^2\left(-1-\frac{\mu r_{*}}{2r}+\frac{r_{*}(1-\mu)}{r}\right)-\frac{tr_{*}}{2r}(1-\mu)\Box |\alpha|^2\\
&+\frac{tr_{*}}{r}(1-\mu)V|\alpha|^2-\frac{1}{2}V\nabla_{b} Z^{b}|\alpha|^2,
\end{align*}

$$
\nabla^{a} T_{ab} Z^{b}=-\frac{1}{2}Z(V)|\alpha|^2,
$$

\noindent
again using the identity $\Box{|\alpha|^2} = 2V|\alpha|^2 + 2\slashed{\nabla}^{c}\alpha\slashed{\nabla}_{c}\bar{\alpha}$.

In total, we find
\begin{align}
\begin{split}
K^Z&= t|\tilde{\slashed \nabla}\alpha|^2\left(-1-\frac{\mu r_{*}}{2r}+\frac{r^*(1-\mu)}{r}\right)-\frac{tr_{*}}{2r}(1-\mu)\Box |\alpha|^2\\
&+\left(\frac{tr_{*}}{r}(1-\mu)V-\frac{1}{2}V\nabla_{b} Z^{b}-\frac{1}{2}Z(V) \right)|\alpha|^2.
\end{split}
\end{align}

As with the $X$ multiplier, we apply integration by parts, encoded by a weight function $\omega^{Z}$, to swap the d'Alembertian.  Namely, we let 

\begin{equation}
\omega^Z:=\frac{2tr^*}{r}(1-\mu)
\end{equation}

\noindent
and define the weighted energy current

\begin{equation}
J^{Z,\omega^Z}_{a} :=J^Z_{a}+\frac{1}{4}\omega^{Z}\nabla_{a} |\alpha|^2-\frac{1}{4}|\alpha|^2\nabla_{a}\omega^{Z},
\end{equation}

\noindent
and the weighted density

\begin{equation}
K^{Z,\omega^{Z}} := K^{Z} + \frac{1}{4}\omega^{Z}\Box |\alpha|^2 -\frac{1}{4} \Box \omega^{Z} |\alpha|^2.
\end{equation}

With our choice of weight function, we compute

\begin{align}
\begin{split}
K^{Z,\omega^Z}&= t|\tilde{\slashed \nabla}\alpha|^2\left(-1-\frac{\mu r_{*}}{2r}+\frac{r_{*}(1-\mu)}{r}\right)-
\frac{1}{4}\Box\omega^Z |\alpha|^2\\
&+\left(\frac{tr_{*}}{r}(1-\mu)V-\frac{1}{2}V\nabla_{b} Z^{b}-\frac{1}{2}Z(V) \right)|\alpha|^2.
\end{split}
\end{align}

As in the earlier work \cite{DRClay}, the first two terms are non-negative for $r$ near the event horizon and near infinity.  Likewise, the coefficient
 
$$\frac{tr_{*}}{r}(1-\mu)V-\frac{1}{2}V\nabla_{b} Z^{b}-\frac{1}{2}Z(V)=\frac{4Mt(r-2M)}{r^5}\left((2r-8M)\log \frac{r-2M}{M}-7r+12M \right)$$

\noindent
is non-negative in these regimes. Hence there exist $2M<r'<R'<\infty$, hereafter regarded as fixed, such that

\begin{align}\label{Zdensity}
\begin{split}
K^{Z,\omega^Z}&\geq 0\ \textup{as}\ r\leq r'\ \textup{or}\ r\geq R',\\
\Big{|}K^{Z,\omega^Z}\Big{|}&\leq Ct\left(|\alpha|^2+|\slashed\nabla\alpha|^2\right)\ \textup{as}\ r'\leq r\leq R'.
\end{split}
\end{align}

The density estimate above will feature in an application of Stokes' theorem on the spacetime region $\{\tau' \leq t \leq \tau\}$.  Before carrying out this computation, however, we estimate as well the weighted boundary integrals.

We define the the unweighted $Z$-energy by

\begin{equation}
E^{Z}_{\alpha}(\tau) := \int_{\{ t = \tau \}} J^{Z}_{a}\eta^{a}_{\tau}
\end{equation}

\noindent
and the weighted $Z$-energy by

\begin{equation}
E^{Z,\omega^{Z}}_{\alpha}(\tau) := \int_{\{ t = \tau \}} J^{Z,\omega^{Z}}_{a}\eta^{a}_{\tau}.
\end{equation}

Note that the future-directed unit normal is given by $\eta_{\tau}^{a} =\frac{1}{\sqrt{1-\mu}}\partial_t=\frac{1}{2\sqrt{1-\mu}}(\partial_u+\partial_v),$ such that 

\begin{equation}
E^{Z}_{\alpha}(\tau) = \int_{\{t = \tau \}} \frac{1}{2\sqrt{1-\mu}}\left( u^2|\slashed{\nabla}_{u}\alpha|^2+v^2|\slashed{\nabla}_{v}\alpha|^2+(u^2+v^2)(1-\mu)|\tilde{\slashed\nabla}\alpha|^2+(u^2+v^2)(1-\mu)V|\alpha|^2 \right),
 \end{equation}

\noindent
and

\begin{align}
\begin{split}
E^{Z,\omega^{Z}}_{\alpha}(\tau) = & E^{Z}_{\alpha}(\tau) + \int_{\{t = \tau\}} \frac{1}{4}\omega^{Z}\nabla_{a} |\alpha|^2 \eta_{\tau}^{a}-\frac{1}{4}|\alpha|^2\nabla_{a}\omega^Z \eta_{\tau}^{a} \\
= &  E^{Z}_{\alpha}(\tau) + \int_{\{t = \tau\}}\frac{1}{2\sqrt{1-\mu}}\left(2t\frac{r_{*}}{r}(1-\mu) \alpha\cdot{\slashed{\nabla}_{t}\alpha}-\frac{r_{*}}{r}(1-\mu)|\alpha|^2 \right).
\end{split}
\end{align}

The weighted energy can be rewritten using integration by parts.  In particular, introducing the analog of the Minkowski scaling field

\begin{align}\label{Soperators}
S = & 2(t\slashed{\nabla}_t+r_{*}\slashed{\nabla}_{r_{*}}) = 2(v\slashed{\nabla}_{v} + u\slashed{\nabla}_{u}),\\
\underline{S} = & 2(t\slashed{\nabla}_{r_{*}}+r_{*}\slashed{\nabla}_t) = 2(v\slashed{\nabla}_{v} - u \slashed{\nabla}_{u}),
\end{align}

\noindent
we rewrite in terms of these operators and integrate by parts with respect to the Regge-Wheeler $r_{*}$ coordinate to find

\begin{align*}
&\int_{\{t=\tau\}} \frac{1}{4}\omega^Z\nabla_{a}|\alpha|^2 \eta_{\tau}^{a} - \frac{1}{4}|\alpha|^2\nabla_{a}\omega^Z \eta_{\tau}^{a}\\
=&\int_{S^2}d\sigma \int_{-\infty}^\infty dr_{*} \left[\frac{r^2}{2}\left(2t\frac{r_{*}}{r}(1-\mu)\alpha\cdot \slashed{\nabla}_{t}\alpha-\frac{r_{*}}{r}(1-\mu)|\alpha|^2 \right)\right]\Big{|}_{t = \tau}\\
=&\int_{S^2}d\sigma\int_{-\infty}^\infty dr_{*} \left[rr_{*}(1-\mu)\alpha\cdot\left(\frac{S\alpha}{2}-r_{*}\slashed{\nabla}_{r_{*}}\alpha\right)-\frac{rr_{*}}{2}(1-\mu)|\alpha|^2\right]\Big{|}_{t = \tau}\\
=&\int_{S^2}d\sigma \int_{-\infty}^\infty dr_{*} \left[ \frac{rr_{*}(1-\mu)}{2}S\alpha\cdot\alpha+\frac{1}{2}|\alpha|^2\left(2r_{*}r(1-\mu)+(1-\mu)(r_{*})^2\right)-\frac{rr_{*}}{2}(1-\mu)|\alpha|^2\right]\Big{|}_{t = \tau}\\
=&\int_{S^2}d\sigma \int_{-\infty}^\infty dr_{*}\left[ \frac{1}{2}r^2(1-\mu)\left(\frac{r_{*}}{r}S\alpha \cdot \alpha+\left(\frac{r_{*}}{r}\right)^2|\alpha|^2+\frac{r_{*}}{r}|\alpha|^2\right)\right]\Big{|}_{t = \tau}\\
=&\int_{S^2}d\sigma \int_{-\infty}^\infty dr_{*} \left[ \frac{1}{2}r^2(1-\mu)\left( \left|\frac{1}{2}S\alpha+\frac{r_{*}}{r}\alpha\right|^2-\frac{1}{4}|S\alpha|^2+\frac{r_{*}}{r}|\alpha|^2 \right)\right]\Big{|}_{t = \tau}.
\end{align*}

By similar computation,
\begin{align*}
&\int_{\{t=\tau\}}\frac{1}{4}\omega^Z\nabla_{a}|\alpha|^2 \eta_{\tau}^{a} -\frac{1}{4}|\alpha|^2\nabla_{a}\omega^Z \eta_{\tau}^{a}\\
=&\int_{S^2}d\sigma \int_{-\infty}^\infty dr_{*} \left[\frac{1}{2}r^2(1-\mu)\left( \left|\frac{1}{2}\underline{S}\alpha+\frac{t}{r}\alpha\right|^2-\frac{1}{4}|\underline{S}\alpha|^2-\frac{r_{*}}{r}|\alpha|^2 \right)\right]\Big{|}_{t = \tau}.
\end{align*}

Combining the two computations, we rewrite the weight term as
\begin{align*}
&\int_{\{t=\tau\}}\frac{1}{4}\omega^Z\nabla_{a}|\alpha|^2 \eta_{\tau}^{a}-\frac{1}{4}|\alpha|^2\nabla_{a}\omega^Z \eta_{\tau}^{a}\\
=&\int_{S^2}d\sigma \int_{-\infty}^\infty dr_{*}\left[ \frac{1}{2}r^2(1-\mu)\left(\frac{1}{2}\left|\frac{1}{2}S\alpha+\frac{r_{*}}{r}\alpha\right|^2+\frac{1}{2}\left|\frac{1}{2}\underline{S}\alpha+\frac{t}{r}\alpha\right|^2-\frac{1}{8}|S\alpha|^2-\frac{1}{8}|\underline{S}\alpha|^2\right)\right]\Big{|}_{t = \tau}\\
=&\int_{\{t=\tau\}}\frac{1-\mu}{2\sqrt{1-\mu}}\left(\frac{1}{2}\left|\frac{1}{2}S\alpha+\frac{r_{*}}{r}\alpha\right|^2+\frac{1}{2}\left|\frac{1}{2}\underline{S}\alpha+\frac{t}{r}\alpha\right|^2-\frac{1}{8}|S\alpha|^2-\frac{1}{8}|\underline{S}\alpha|^2\right).
\end{align*}

Expressed in terms of $S$ and $\underline{S}$, the unweighted $Z$-energy has the form

$$ 
E^{Z}_{\alpha}(\tau) = \int_{\{t=\tau\}} \frac{1}{2\sqrt{1-\mu}}\left(\frac{1}{8}|S\alpha|^2+\frac{1}{8}|\underline{S}\alpha|^2+(u^2+v^2)(1-\mu)|\tilde{\slashed \nabla}\alpha|^2+(u^2+v^2)(1-\mu)V|\alpha|^2 \right).
$$

Taken together, these two computations allow us to rewrite the weighted $Z$-energy as

\begin{align*}
E^{Z,\omega^{Z}}_{\alpha}(\tau) = &\int_{\{t=\tau\}}J^{Z,\omega^Z}_{a} \eta^{a}_{\tau}\\
=&\int_{\{t=\tau\}} \frac{1}{2\sqrt{1-\mu}}\left(\frac{1}{8}|S\alpha|^2+\frac{1}{8}|\underline{S}\alpha|^2+(u^2+v^2)(1-\mu)|\tilde{\slashed \nabla}\alpha|^2+(u^2+v^2)(1-\mu)V|\alpha|^2 \right)\\
+&\int_{\{t=\tau\}}\frac{1-\mu}{2\sqrt{1-\mu}}\left(\frac{1}{2}\left|\frac{1}{2}S\alpha+\frac{r_{*}}{r}\alpha\right|^2+\frac{1}{2}\left|\frac{1}{2}\underline{S}\alpha+\frac{t}{r}\alpha\right|^2-\frac{1}{8}|S\alpha|^2-\frac{1}{8}|\underline{S}\alpha|^2\right)\\
=&\int_{\{t=\tau\}}\frac{1}{2\sqrt{1-\mu}}\left(\frac{\mu}{8}|S\alpha|^2+\frac{\mu}{8}|\underline{S}\alpha|^2+(u^2+v^2)(1-\mu)|\tilde{\slashed \nabla}\alpha|^2+(u^2+v^2)(1-\mu)V|\alpha|^2 \right)\\
+&\int_{\{t=\tau\}}\frac{1}{2\sqrt{1-\mu}}\left(\frac{1-\mu}{2}\left|\frac{1}{2}S\alpha+\frac{r_{*}}{r}\alpha\right|^2+\frac{1-\mu}{2}\left|\frac{1}{2}\underline{S}\alpha+\frac{t}{r}\alpha\right|^2\right).
\end{align*}

Recall that our $\mathcal{L}(-2)$ section $\alpha$ satisfies a Poincar\'{e} inequality on the sphere,

\begin{align*}
\int_{S^2(r)}(u^2+v^2)|\tilde{\slashed\nabla}\alpha|^2&\geq \int_{S^2(r)} \frac{2(u^2+v^2)}{r^2}|\alpha|^2\\
                                                  &=\int_{S^2(r)}\left(\frac{t^2}{r^2}+\frac{r_{*}^2}{r^2}\right)|\alpha|^2.
\end{align*} 

\noindent
This inequality allows us to control the cross-terms appearing in the expression above, yielding the estimate

\begin{equation}\label{eq: Zcomparison}
E^{Z,\omega^{Z}}_{\alpha}(\tau) \geq c E^{Z}_{\alpha}(\tau),
\end{equation}

\noindent
controlling the unweighted $Z$-energy by the weighted $Z$-energy.

Having estimated both the bulk and the boundary terms, we apply Stokes' theorem in the region $\{ \tau' \leq t \leq \tau \}$, with the goal of proving uniform boundedness of the unweighted $Z$-energy $E^{Z}_{\alpha}$.  We find, for $\tau \geq \tau' \geq 0$,

\begin{align*}
\int_{\{t=\tau\}}J^{Z,\omega^Z}_{a} \eta^{a}_{\tau} &=\int_{\{t=\tau'\}}J^{Z,\omega^Z}_{a} \eta^{a}_{\tau'}-\int_{\{\tau'\leq t\leq \tau\}}K^{Z,\omega^Z}\\
                                          &\leq \int_{\{t=\tau'\}}J^{Z,\omega^Z}_{a} \eta^{a}_{\tau'}+\int_{\{\tau'\leq t\leq \tau\}\cap\{r'\leq r\leq R'\}}|K^{Z,\omega^Z}|\\
                                          &\leq \int_{\{t=\tau'\}}J^{Z,\omega^Z}_{a} \eta^{a}_{\tau'}+\tau C\int_{\{\tau'\leq t\leq \tau\}\cap\{r'\leq r\leq R'\}}\left[K^{X,\omega^X}[\alpha]+\sum_{i=1}^{3}K^{X,\omega^X}[\Omega_{i}\alpha]\right]\\
                                          &\leq \int_{\{t=\tau'\}}J^{Z,\omega^Z}_{a} \eta^{a}_{\tau'}+\tau C\int_{\{t=\tau'\}}\left[J^{T}_{a}[\alpha]\eta^{a}_{\tau'}+\sum_{i =1}^{3}J^{T}_{a}[\Omega_{i}\alpha]\eta^{a}_{\tau'}\right],
\end{align*}

\noindent
or

\begin{equation}\label{eq: weightedZenergy}
E^{Z,\omega^{Z}}_{\alpha}(\tau) \leq  E^{Z,\omega^{Z}}_{\alpha}(\tau') + \tau C \left[ E^{T}_{\alpha}(\tau') + \sum_{i = 1}^{3}E^{T}_{\Omega_{i}\alpha}(\tau')\right],
\end{equation}

\noindent
where we have used the density estimates (\ref{Zdensity}) and Lemma \ref{DegIntegratedDecayA} for $K^{Z,\omega^{Z}}$ and $K^{X,\omega^{X}}$, respectively.  Uniform boundedness of the unweighted $Z$-energy is established by means of a localized decay estimate and bootstrapping, as follows.

\begin{lemma}\label{localized}
Suppose $\alpha$ is a solution of (\ref{eq: RW1}), with appropriate regularity, and let $\tau_0>0$ and $\tau_1=1.1\tau_0$. For any pair $2M < r_0 < R_0 < \infty$ satisfying $|r_*(r_0)|,|r_*(R_0)|\leq 0.1\tau_0$, we have

\begin{align*}
\int_{\{\tau_0\leq t \leq \tau_1\}\cap\{r_0\leq r\leq R_0\}}K^{X,\omega^X} \leq C\tau_0^{-2}\int_{\{t=\tau_0\}}J^{Z,\omega^Z}_{a} \eta^{a}_{\tau_0},
\end{align*}
\noindent
where the constant $C$ is independent of $r_0, R_0,$ and $t_0$.
\end{lemma}

\begin{proof}
Suppose $\alpha|_{t=\tau_0}$ is supported in $-\frac{1}{2}\tau_0\leq r_{*}\leq \frac{1}{2}\tau_0$. By Lemma \ref{DegIntegratedDecayA},

\begin{align*}
\int_{\{\tau_0\leq t\leq \tau_1\}}K^{X,\omega^X}\leq C\int_{\{t=\tau_0\}}J^T_{a} \eta^{a}_{\tau_0} = C\int_{\{t=\tau_0\}\cap{\{\frac{1}{2}\tau_0\leq r_{*}\leq \frac{1}{2}\tau_0\}}}J^T_{a} \eta^{a}_{\tau_0}.
\end{align*}

In the region $\{t = \tau_0\} \cap{\{ -\frac{1}{2}\tau_0\leq r_{*}\leq \frac{1}{2}\tau_0\}}$, the Eddington-Finkelstein coordinates $u$ and $v$ satisfy $u=\frac{1}{2}(t-r_{*})\geq \frac{1}{4}\tau_0$, $v=\frac{1}{2}(t+r_{*})\geq \frac{1}{4}\tau_0$, such that

$$ \int_{\{t = \tau_0\}} J^{Z,\omega^{Z}}_{a}\eta^{a}_{\tau_0} \geq c \int_{\{t = \tau_0\}} J^{Z}_{a}\eta^{a}_{\tau_0} \geq c \tau_{0}^2 \int_{\{t = \tau_0\}} J^{T}_{a}\eta^{a}_{\tau_0},$$

\noindent
where we have used \eqref{eq: Zcomparison} and the fact that our region is away from the event horizon.  Taken together, these two estimates give the localized 
decay

\begin{align*}
\int_{\{\tau_0\leq t\leq \tau_1\}}K^{X,\omega^X}&\leq C\tau_0^{-2}\int_{\{t=\tau_0\}}J^{Z}_{a}\eta^{a}_{\tau_0}\\
                                      &\leq C\tau_0^{-2}\int_{\{t=\tau_0\}}J^{Z,\omega^Z}_{a} \eta^{a}_{\tau_0}.
\end{align*}

In general, fixing a bump function $\chi:\mathbb{R} \mapsto \mathbb{R}$, supported in  $[-\frac{1}{2},\frac{1}{2}]$ and with $\chi= 1$ in $[-\frac{1}{4},\frac{1}{4}]$, we define $\tilde{\alpha}$ to be the solution of \eqref{eq: RW1} with initial data $\tilde{\alpha}|_{t=\tau_0}=\chi(r_{*}/\tau_0)\alpha|_{t=\tau_0}$ and $\mathcal{L}_{T}\tilde{\alpha}|_{t=\tau_0}=\chi(r_{*}/\tau_0)\mathcal{L}_{T}\alpha|_{t=\tau_0}$. By finite speed of propagation, such a specification gives $\alpha=\tilde{\alpha}$ throughout the spacetime region $\{t_0\leq t\leq t_1\}\cap\{r_0\leq r\leq R_0\}$, so that

\begin{align*}
\int_{\{\tau_0\leq t \leq \tau_1\}\cap\{r_0\leq r\leq R_0\}}K^{X,\omega^X}[\alpha]&=\int_{\{\tau_0\leq t \leq \tau_1\}\cap\{r_0\leq r\leq R_0\}}K^{X,\omega^X}[\tilde{\alpha}]\\
                                                             &\leq C\tau_0^{-2}\int_{\{t=\tau_0\}}J^{Z,\omega^Z}_{a} [\tilde{\alpha}] \eta^{a}_{\tau_0}.
\end{align*}

We compute
\begin{align*}
&J^{Z,\omega^Z}_{a} [\tilde{\alpha}] \eta^{a}_{\tau_0}\\
=&\frac{1}{2\sqrt{1-\mu}}\left( (u^2+v^2)|\slashed{\nabla}_{t}\tilde{\alpha}|^2+(u^2+v^2)|\slashed{\nabla}_{r_{*}}\tilde{\alpha}|^2+(u^2+v^2)(1-\mu)|\tilde{\slashed\nabla}\tilde{\alpha}|^2+(u^2+v^2)(1-\mu)V|\tilde{\alpha}|^2 \right)\\
=&\frac{\chi^2}{2\sqrt{1-\mu}}\left( (u^2+v^2)|\slashed{\nabla}_{t}\alpha|^2+(u^2+v^2)|\slashed{\nabla}_{r_{*}}\alpha|^2+(u^2+v^2)(1-\mu)|\tilde{\slashed\nabla}\alpha|^2+(u^2+v^2)(1-\mu)V|\alpha|^2 \right)\\
+&\frac{1}{2\sqrt{1-\mu}} \left((u^2+v^2)(\frac{2}{\tau_0}\chi\chi'\alpha \cdot \slashed{\nabla}_{r_{*}}\alpha+\frac{\chi'^2}{\tau_0^2}|\alpha|^2)  \right).
\end{align*}

Applying Young's inequality, the lemma is proved if we demonstrate

\begin{align*}
\int_{\{\tau=\tau_0\} \cap \{-\frac{1}{2}\tau_0\leq r_{*}\leq\frac{1}{2}\tau_0\}} \frac{1}{2\sqrt{1-\mu}}\frac{(u^2+v^2)}{\tau_0^2}|\alpha|^2\leq C\int_{\{t=t_0\}}J^{Z,\omega^Z}_{a} [\alpha]\eta^{a}_{\tau_0},
\end{align*}

We will use the one dimensional inequality

\begin{align*}
\int_{-\tau}^\tau |f(x)|^2dx\leq C\tau^2\left(\int_{-\tau}^\tau |f'(x)|^2 dx+\int_{-1}^1 |f(x)|^2 dx \right),
\end{align*}

\noindent
to estimate

\begin{align*}
&\int_{\{t=\tau_0\}\cap \{-\frac{1}{2}\tau_0\leq r_{*}\leq\frac{1}{2}\tau_0\}}\frac{1}{2\sqrt{1-\mu}}\frac{(u^2+v^2)}{t_0^2}|\alpha|^2\\
\leq&C\int_{-\frac{1}{2}\tau_0}^{\frac{1}{2}\tau_0}dr_{*} \int_{S^2} d\sigma\ r^2|\alpha|^2\\
\leq&C\tau_0^2 \left(\int_{-\frac{1}{2}\tau_0}^{\frac{1}{2}\tau_0}dr_{*} \int_{S^2} d\sigma\ \left(|\slashed{\nabla}_{r_{*}}\alpha|^2+\frac{1-\mu}{r^2}|\alpha|^2\right)r^2+\int_{-1}^1dr_{*}\int_{S^2}d\sigma\ r^2|\alpha|^2\right)\\
=&C\tau_0^2\left(\int_{\{t=\tau_0\}\cap\{ \frac{1}{2}\tau_0\leq r_{*}\leq\frac{1}{2}\tau_0\}} \frac{1}{\sqrt{1-\mu}}\left(|\slashed{\nabla}_{r_{*}}\alpha|^2+\frac{1-\mu}{r^2}|\alpha|^2\right)+\int_{\{t=\tau_0\}\cap\{ -1\leq r_{*}\leq 1\}} \frac{1}{\sqrt{1-\mu}}|\alpha|^2\right)\\
\leq &C\int_{\{t=\tau_0\}}J^{Z,\omega^Z}_{a} [\alpha]\eta^{a}_{\tau_0}.
\end{align*}
\end{proof}

With this localized decay estimate in hand, we adapt the bootstrapping method of \cite{DR} to deduce uniform boundedness of the weighted $Z$-energy, as follows.
  
\begin{lemma}\label{weightedZbound}
Suppose $\alpha$ is a solution of the Regge-Wheeler equation (\ref{eq: RW1}), smooth and compactly supported on $\{ t = 0 \}$. Then the weighted $Z$-energy associated with $\alpha$, $E^{Z,\omega^Z}_{\alpha}(\tau)$, is bounded for $\tau \geq 0$.  Indeed, defining the initial energy

\begin{equation}\label{eq: E1}
E_1[\alpha] := \sum_{(m)\leq 3} \int_{\{t = 0 \}} r^2J^{N}_{a}[\Omega^{(m)}\alpha]\eta^{a},
\end{equation}

\noindent
where $(m)$ is a multi-index and $\Omega = \{ \Omega_{i} | i = 1,2,3\}$, we find that 

\begin{equation}
E^{Z}_{\alpha}(\tau) \leq CE_1[\alpha]
\end{equation}

\noindent
for all $\tau \geq 0.$

\end{lemma}

\begin{proof}
Recall the fixed $2M < r' < R' < \infty$ determined in the density estimate (\ref{Zdensity}).  With such a fixed pair $(r', R')$, we take $\tau_0>0$ large such that $|r_{*}(r')|, |r_{*}(R')|\leq 0.1\tau_0$ and set $\tau_{i+1}=1.1\tau_{i}$. 

In the proof, we employ a simplifying notation with respect to the angular moment operators $\Omega_{i}$.  Namely, when we write simply $\Omega$, with no subscript, it is understood that summation over all three operators is intended.  Likewise, we write $\Omega^2$ and $\Omega^3$ to encode summation over all second and third derivatives, respectively.

From \eqref{eq: weightedZenergy}, we have

\begin{align*}
E^{Z,\omega^Z}_\alpha(\tau_{i})&\leq E^{Z,\omega^Z}_\alpha(\tau_0)+C\tau_{i}\left(E^{T}_\alpha(\tau_0)+ E^{T}_{\Omega\alpha}(\tau_0) \right),\\                                   
E^{Z,\omega^Z}_{\Omega\alpha}(\tau_i)&\leq E^{Z,\omega^Z}_{\Omega\alpha}(\tau_0)+C\tau_{i}\left(E^{T}_{\Omega\alpha}(\tau_0)+E^{T}_{\Omega^2\alpha}(\tau_0)\right).\\
\end{align*}

Let $\mathcal{R}_j=\{\tau_j\leq t\leq \tau_{j+1}\}\cap\{r'\leq r\leq R'\}$. Applying the localized decay estimate, Lemma \ref{localized}, we have

\begin{align*}
\int_{\mathcal{R}_j} K^{X,\omega^X}[\alpha]&\leq C\tau_j^{-2}E^{Z,\omega^Z}_\alpha(\tau_j)\\
                                         &\leq C\tau_j^{-2}E^{Z,\omega^Z}_\alpha(\tau_0)+C\tau_j^{-1}\left(E^{T}_\alpha(\tau_0)+E^{T}_{\Omega\alpha}(\tau_0) \right),\\
\int_{\mathcal{R}_j} K^{X,\omega^X}[\Omega\alpha]&\leq C\tau_{j}^{-2}E^{Z,\omega^Z}_{\Omega\alpha}(\tau_0)+C\tau_j^{-1}\left(E^{T}_{\Omega\alpha}(\tau_0)+E^{T}_{\Omega^2\alpha}(\tau_0) \right).\\
\end{align*}

Reinserting into the weighted $Z$-energy estimate, we find

\begin{align*}
E^{Z,\omega^Z}_\alpha(\tau_i)\leq &E^{Z,\omega^Z}_\alpha(\tau_0)+C\sum_{j=0}^{i-1}\int_{\mathcal{R}_j}t\left(K^{X,\omega^X}[\alpha]+K^{X,\omega^X}[\Omega\alpha]\right)\\
\leq & E^{Z,\omega^Z}_\alpha(\tau_0)+C\sum_{j=0}^{i-1} \left[\tau_i^{-1}\left( E^{Z,\omega^Z}_\alpha(\tau_0)+E^{Z,\omega^Z}_{\Omega\alpha}(\tau_0) \right)+\left(E^{T}_{\alpha}(\tau_0) + E^{T}_{\Omega\alpha}(\tau_0) + E^{T}_{\Omega^2\alpha}(\tau_0) \right)\right]\\
\leq &C\left(E^{Z,\omega^Z}_\alpha(\tau_0)+E^{Z,\omega^Z}_{\Omega\alpha}(\tau_0)\right)+C\log \tau_i\left(E^{T}_{\alpha}(\tau_0) + E^{T}_{\Omega\alpha}(\tau_0) + E^{T}_{\Omega^2\alpha}(\tau_0) \right),
\end{align*}

\noindent
where we rely upon $\tau_{i} = (1.1)^{i-1}\tau_0$, implying in particular $\log \tau_i \sim i$.

Continuing the bootstrap, we insert this improvement back into our localized decay estimate:

\begin{align*}
\int_{\mathcal{R}_{j}} K^{X,\omega^X}[\alpha]&\leq C\tau_j^{-2}E^{Z,\omega^Z}_\alpha(\tau_j)\\
                                         &\leq C\tau_j^{-2}\left( E^{Z,\omega^Z}_\alpha(\tau_0)+E^{Z,\omega^Z}_{\Omega\alpha}(\tau_0)\right)+C\tau_j^{-2}\log \tau_j \left(E^{T}_{\alpha}(\tau_0) + E^{T}_{\Omega\alpha}(\tau_0) + E^{T}_{\Omega^2\alpha}(\tau_0) \right),\\
\int_{\mathcal{R}_{j}} K^{X,\omega^X}[\Omega\alpha]&\leq C\tau_j^{-2}\left( E^{Z,\omega^Z}_{\Omega\alpha}(\tau_0)+E^{Z,\omega^Z}_{\Omega^2\alpha}(\tau_0)\right)+C\tau_j^{-2}\log \tau_j \left(E^{T}_{\Omega\alpha}(\tau_0) + E^{T}_{\Omega^2\alpha}(\tau_0) + E^{T}_{\Omega^3\alpha}(\tau_0) \right).
\end{align*}

Finally, we deduce the uniform bound on the sequence $\{ \tau_{i} \}$

\begin{align*}
E^{Z,\omega^Z}_\alpha(\tau_i) \leq &E^{Z,\omega^Z}_\alpha(\tau_0)+C\sum_{j=0}^{i-1}\int_{\mathcal{R}_j}t\left(K^{X,\omega^X}[\alpha]+K^{X,\omega^X}[\Omega\alpha]\right)\\
\leq &E^{Z,\omega^Z}_\alpha(\tau_0)+C\sum_{j=0}^{i-1} \tau_j^{-1} \left( E^{Z,\omega^Z}_{\alpha}(\tau_0)+E^{Z,\omega^Z}_{\Omega\alpha}(\tau_0)+E^{Z,\omega^Z}_{\Omega^2\alpha}(\tau_0)\right)	\\
+&C\sum_{j=1}^{i-1} \tau_j^{-1}\log \tau_j \left(E^{T}_{\alpha}(\tau_0) + E^{T}_{\Omega\alpha}(\tau_0) + E^{T}_{\Omega^2\alpha}(\tau_0) + E^{T}_{\Omega^3\alpha}(\tau_0)\right)\\
\leq &C\left( E^{Z,\omega^Z}_{\alpha}(\tau_0)+E^{Z,\omega^Z}_{\Omega\alpha}(\tau_0)+E^{Z,\omega^Z}_{\Omega^2\alpha}(\tau_0)+\left(E^{T}_{\alpha}(\tau_0) + E^{T}_{\Omega\alpha}(\tau_0) + E^{T}_{\Omega^2\alpha}(\tau_0) + E^{T}_{\Omega^3\alpha}(\tau_0)\right)\right)
\end{align*}

The extension to all large $\tau \geq \tau_0$ is obvious, and as well for those $0 \leq \tau < \tau_0$ in a temporally compact region.
\end{proof}

\subsection{Uniform Decay}
\subsubsection{Integrated Decay}

We define the family of spacelike hypersurfaces $\tilde{\Sigma}_{\tau}$ by the conditions

\begin{align}\label{decayFoliation}
\begin{split}
\tau &= t+2M\log(r-2M)+c_0\textup{, for}\ r\leq 3M,\\
 &= t-\sqrt{r^2+1}+c_1\textup{, for}\ r \geq 20M,
\end{split}
\end{align}

\noindent
with the specification in the spatially compact region $3M < r < 20M$ and the choice of constants $c_0$ and $c_1$ made in such a way that $u,v\geq \tau$ on $\tilde{\Sigma}_\tau$.  Further, we define the $N$-energy on $\tilde{\Sigma}_{\tau}$ by

\begin{equation}
E^{N}_{\alpha}(\tilde{\Sigma}_{\tau}) := \int_{\tilde{\Sigma}_\tau}J^N_{a} \eta^{a}_{\tilde{\Sigma}_{\tau}}.
\end{equation}

\begin{theorem}\label{NondegIntegratedDecayA}
Suppose $\alpha$ is a solution of (\ref{eq: RW1}), smooth and compactly supported at $\{ t = 0 \}$.  Then $\alpha$ satisfies the integrated decay estimate

\begin{equation}
E^{N}_{\alpha}(\tilde{\Sigma}_{\tau})\leq CE_1[\alpha]\tau^{-2},
\end{equation}

\noindent
uniformly on the hypersurfaces $\tilde{\Sigma}_{\tau}$ \eqref{decayFoliation} defined above, with the initial energy $E_1$ defined as before (\ref{eq: E1}).
\end{theorem}

\begin{proof}
Having defined $\tilde{\Sigma}_{\tau}$ so that $\inf_{\tilde{\Sigma}_{\tau}}u\geq \tau$ and $\inf_{\tilde{\Sigma}_{\tau}}v \geq \tau$ , we have for any $s > 0$

\begin{align*}
CE_1&\geq \int_{\{t=s\}\cap J^+(\tilde{\Sigma}_\tau)} J^{Z,\omega^Z}_{a} \eta^{a}_{s} \\
&\geq c\tau^2\int \int_{\{t=s\}\cap J^+(\tilde{\Sigma}_\tau)} J^T_{a} \eta^{a}_{s} = c\tau^2\int_{\tilde{\Sigma}_\tau\cap J^{-}(\{t=s\})} J^T_{a} \eta^{a}_{\tilde{\Sigma}_\tau},
\end{align*}

\noindent
where we have used the uniform boundedness of the weighted $Z$-energy, Lemma \ref{weightedZbound}, and the comparison between weighted and unweighted $Z$-energies \eqref{eq: Zcomparison}.

Taking $s\to\infty$ we have 

\begin{align*}
\int_{\tilde{\Sigma}_\tau} J^T_{a} \eta^{a}_{\tilde{\Sigma}_\tau}\leq CE_1\tau^{-2}.
\end{align*}

It remains to upgrade to an integrated decay estimate for the non-degenerate red-shift energy.  We remind the reader that the details of the red-shift multiplier $N$ appear in Appendix C; in particular, $N$ satisfies

 \begin{align*}
 K^{N} &\geq{c J^{N}_{a}N^{a}},\ &&\textup{for}\ 2M\leq r \leq r_0,\\
 J^{N}_{a}N^{a} &\sim J^{T}_{a}T^{a}, \ &&\textup{for}\ r_0 \leq r \leq R_0,\\
 |K^{N}| &\leq  C|J^{T}_{a}T^{a}|, \ &&\textup{for}\ r_0 \leq r \leq R_0,\\
 N &= T, \ &&\textup{for}\ r\geq R_0,
 \end{align*}
 
 \noindent
 where $2M < r_0 < R_0 < \infty$ can be regarded as fixed, say $r_0 = 3M$ and $R_0 = 10M$ in light of our hypersurface definition \eqref{decayFoliation}.

In what follows, we denote by $\mathcal{R}(\tau', \tau)$ the spacetime region between the hypersurfaces $\tilde{\Sigma}_{\tau'}$ and $\tilde{\Sigma}_{\tau}$, with $0\leq{\tau'}\leq \tau$.  Concretely, $\mathcal{R}(\tau', \tau) = (J^{+}(\tilde{\Sigma}_{\tau'})\setminus J^{+}(\tilde{\Sigma}_{\tau}))
\cap J^{-}(\mathcal{I}^{+})$.

Applying the divergence theorem in the spacetime region $\mathcal{R}(\tau', \tau)$, we deduce

\begin{align*}
&\int_{\tilde{\Sigma}_{\tau}\cap\{2M\leq r\leq r_0\}}J^{N}_{a} \eta^{a}_{\tilde{\Sigma}_{\tau}}+\int_{\tilde{\Sigma}_{\tau}\cap\{r\geq r_0\}}J^{N}_{a} \eta^{a}_{\tilde{\Sigma}_{\tau}}+\int_{\mathcal{R}(\tau',\tau)\cap\{2M\leq r\leq r_0\}}K^{N}\\
\leq &\int_{\tilde{\Sigma}_{\tau'}\cap\{2M\leq r\leq r_0\}}J^{N}_{a} \eta^{a}_{\tilde{\Sigma}_{\tau'}}+\int_{\tilde{\Sigma}_{\tau'}\cap\{r\geq r_0\}}J^{N}_{a} \eta^{a}_{\tilde{\Sigma}_{\tau'}}+\int_{\mathcal{R}(\tau',\tau)\cap\{r_0\leq r\leq R_0\}}|K^{N}|
\end{align*}

Defining 

\begin{equation}
g(\tau):=\int_{\tilde{\Sigma}_{\tau}\cap\{2M\leq r\leq r_0\}}J^{N}_{a} \eta^{a}_{\tilde{\Sigma}_{\tau}},
\end{equation}

\noindent
and noting that, away from the event horizon

$$
\int_{\tilde{\Sigma}_{\tau}\cap\{r\geq r_0\}}J^{N}_{a} \eta^{a}_{\tilde{\Sigma}_{\tau}}\sim \int_{\tilde{\Sigma}_{\tau}\cap\{r\geq R_0\}}J^{T}_{a} \eta^{a}_{\tilde{\Sigma}_{\tau}}\leq CE_1\tau^{-2}, 
$$

\noindent
we obtain an integral inequality, just as in the earlier discussion of red-shift:

\begin{equation}
g(\tau)+c\int_{\tau'}^{\tau} g(s)ds \leq g(\tau')+CE_1\max\{\tau-\tau',1\}(\tau')^{-2}
\end{equation}

Employing a bootstrap argument, we find $g(\tau)\leq CE_1\tau^{-2}$, and the proof is complete.

\end{proof}

\subsubsection{Pointwise Decay}
The non-degerate integrated decay estimate above can be strengthened to a statement of pointwise decay by commutation with the standard Killing fields and application of various Sobolev embedding theorems, just as in the earlier section on uniform boundedness.

\begin{theorem}\label{pointwisedecayA}
Suppose $\alpha$ is a solution of (\ref{eq: RW1}), smooth and compactly supported at $\{ t = 0 \}$.  We define the initial energy

\begin{equation}\label{E2}
E_2[\alpha] := \sum_{(m)\leq 6}\int_{\{t = 0 \}} r^2J^{N}_{a}[\Omega^{(m)}\alpha]\eta^{a}.
\end{equation}

Expressed in this energy, $\alpha$ satisfies the uniform decay estimate 

\begin{equation}
\sup_{\tilde{\Sigma}_{\tau}} |\alpha| \leq C\sqrt{E_2[\alpha]}\tau^{-1}
\end{equation}

\noindent
on the family of hypersurfaces $\tilde{\Sigma}_{\tau}$ (\ref{decayFoliation}) specified above.
\end{theorem}

\begin{proof}

First, note that the quantity $J^N_{a} \eta^{a}_{\tilde{\Sigma}}$ is non-degenerate away from null infinity, with $J^{N}_{a}\eta^{a}_{\tilde{\Sigma}} \geq c|\nabla_{\tilde{\Sigma}}\alpha|^2$.

Near null infinity, recall that $\tilde{\Sigma}_{\tau}$ is defined such that $\tau \sim t-\sqrt{r^2+1}$, with hyperbolic geometry.  In this regime, we have on $\tilde{\Sigma}_{\tau}$

\begin{align*}
&R  :=t\partial_r+r\partial_t,\\
&\eta_{\tilde{\Sigma}_{\tau}}\sim r\partial_r+t\partial_t,\\
&N=T =\partial_t \sim t\eta_{\tilde{\Sigma}_{\tau}}-rR,
\end{align*}

\noindent
where $R$ can be thought of as a radial derivative along $\tilde{\Sigma}_{\tau}$.  We find the comparison

\begin{align*}
J^{N}_{a}\eta_{\tilde{\Sigma}_{\tau}}^{a}&=T_{ab}\eta_{\tilde{\Sigma}_{\tau}}^{a}T^{b}\\
		      &=tT_{ab}\eta_{\tilde{\Sigma}_{\tau}}^{a}\eta_{\tilde{\Sigma}_{\tau}}^{b}-rT_{ab}\eta_{\tilde{\Sigma}_{\tau}}^{a}R^{b}\\
                       &=\frac{t}{2}\left( |\eta_{\tilde{\Sigma}_{\tau}}\alpha|^2+|R\alpha|^2+|\tilde{\slashed\nabla}\alpha|^2+V|\alpha|^2  \right)-r\left(\eta_{\tilde{\Sigma}_{\tau}}\alpha\cdot R\alpha \right)\\
                       &\geq\frac{t-r}{2}\left(|\eta_{\tilde{\Sigma}_{\tau}}\alpha|^2+|R\alpha|^2 \right) +\frac{t}{2}\left(|\slashed\nabla\alpha|^2+V|\alpha|^2\right)\\
                       &\geq \frac{c}{r}|\nabla_{\tilde{\Sigma}_{\tau}}\alpha|^2.
\end{align*}

Now let $\rho$ be the geodesic radial coordinate on $\tilde{\Sigma}$, with $\rho=1$ on the horizon. As before, we compute
\begin{align*}
|\alpha|\leq &\int_1^\infty |\slashed{\nabla}_{\rho}\alpha|d\rho \leq C \left(\int_1^\infty |\slashed{\nabla}_{\rho}\alpha|^2\rho^2 d\rho\right)^{1/2}\\
      \leq &C\left( \int_1^\infty d\rho\int_{S^2}d\sigma \rho^2\left( |\slashed{\nabla}_{\rho}\alpha|^2+|\slashed{\nabla}_{\rho}\Omega\alpha|^2+|\slashed{\nabla}_{\rho}\Omega^2\alpha|^2 \right) \right)^{1/2}\\
      \leq&C\left(\int_{\tilde{\Sigma}_{\tau}} \frac{\rho^2}{r^2}\left( |\nabla_{\tilde{\Sigma}_{\tau}}\alpha|^2+|\nabla_{\tilde{\Sigma}_{\tau}}\Omega\alpha|^2+|\nabla_{\tilde{\Sigma}_{\tau}}\Omega^2\alpha|^2\right)\right)^{1/2}
\end{align*}

Away from null infinity, $\rho$ and $r$ are comparable. Near null infinity, $r \sim e^\rho$, so that $\rho^2/r^2\leq 1/r$.  Hence
 
\begin{align*}
|\alpha| \leq &C\int_{\tilde{\Sigma}_{\tau}} \frac{\rho^2}{r^2}\left( |\nabla_{\tilde{\Sigma}_{\tau}}\alpha|^2+|\nabla_{\tilde{\Sigma}_{\tau}}\Omega\alpha|^2+|\nabla_{\tilde{\Sigma}_{\tau}}\Omega^2\alpha|^2\right)\\
\leq &C\int_{\tilde{\Sigma}}J^{N}_{a}[\alpha]\eta^{a}_{\tilde{\Sigma}_{\tau}}+J^{N}_{a}[\Omega\alpha]\eta^{a}_{\tilde{\Sigma}_{\tau}}+J^{N}_{a}[\Omega^2\alpha]\eta^{a}_{\tilde{\Sigma}_{\tau}}\\
\leq &CE_2\tau^{-2},
\end{align*}

and the decay estimate $|\alpha|\leq C\sqrt{E_2}\tau^{-1}$ is established.

\end{proof}

We remark that, as in the proofs of uniform boundedness, the argument for uniform decay applies for generic families of spacelike hypersurfaces 
$\tilde{\Sigma}_{\tau}$, in this case to those passing through $\mathcal{H}^{+}$ away from the bifurcation sphere and asymptotic to null infinity.

\section{Decoupling $\beta$}

Recall the governing equations

\begin{align}
R_{t\phi}^{(1)} = 0 &\rightarrow (r^2\beta)_{r} + \frac{r^4}{\Delta}\sqrt{2}r\bar{\partial}\gamma = 0,\label{R0phi}\\
R_{r\phi}^{(1)} = 0 &\rightarrow \sqrt{2}r\bar{\partial}\alpha + \beta_{t},\label{R2phi}\\
R_{\theta\phi}^{(1)} = 0 &\rightarrow (r^2\alpha)_{r} -\frac{r^4}{\Delta}\gamma_{t} = 0,\label{R3phi}\\
d^2\zeta^{(1)} = 0 &\rightarrow \frac{1}{r^2}\sqrt{2}r \partial \beta - \gamma_{r} + \frac{r^2}{\Delta}\alpha_{t} = 0.\label{betaAgain}
\end{align}

Applying the operator $\sqrt{2}r\bar\partial$ to (\ref{betaAgain}) and differentiating (\ref{R2phi}) in $t$ and (\ref{R3phi}) in $r$, we have

\begin{gather*}
\frac{-r^2}{\Delta}\sqrt{2}r\bar\partial\alpha_{t} + \frac{-r^2}{\Delta}\beta_{tt} = 0,\\
\left(\frac{\Delta}{r^4}(r^2\beta)_{r}\right)_{r} + \sqrt{2}r\bar\partial \gamma_{r} = 0,\\
2\bar\partial\partial \zeta - \sqrt{2}r\bar\partial \gamma_{r} + \frac{r^2}{\Delta}\sqrt{2}r\bar\partial{\alpha_{t}} = 0.
\end{gather*}

Adding the three equations together, we isolate $\beta$:

$$\frac{r^2}{\Delta}\beta_{tt} + \left(\frac{\Delta}{r^4}(r^2 \beta)_{r}\right)_{r} + 2\bar\partial\partial \beta = 0,$$

\noindent
or, expanding and relating the two Laplacians,

\begin{equation}
\frac{-r^2}{\Delta}\beta_{tt} + \frac{\Delta}{r^2}\beta_{rr} + \frac{\Delta_{r}}{r^2}\beta_{r} + \Delta_{\mathcal{E}(-2)}\beta = W\beta,
\end{equation}

\noindent
where $W = \frac{1}{r^2}\left(1-\frac{8M}{r}\right)$.  Rewritten in this form, $\beta$ is seen to be a solution of a Regge-Wheeler equation

\begin{equation}\label{eq: RW2}
\slashed{\Box}_{\mathcal{L}(-1)} \beta = W\beta
\end{equation}

\noindent
with potential $W = \frac{1}{r^2}\left(1-\frac{8M}{r}\right)$.

We study \eqref{eq: RW2} by means of the spin-weighted spherical harmonics discussed in Appendix \ref{spinHarmonics}.  In particular, denoting by $\beta_1$ the orthogonal $(L^2)$ projection of $\beta$ on to the first eigenspace associated with the operator $\slashed{\Delta}_{\mathcal{E}(-2)}$, and letting $\beta_{\ell > 1}$ be its orthogonal complement, the equation \eqref{eq: RW2} splits as 

\begin{align}
\slashed{\Box}_{\mathcal{L}(-1)} \beta_1 &= W \beta_1,\label{beta1}\\
\slashed{\Box}_{\mathcal{L}(-1)} \beta_{\ell > 1} &= W \beta_{\ell > 1},
\end{align}

\noindent
permitting a separate study of the two. 

\section{Analysis of $\beta_1$}

Projecting to the lowest mode, we find

\begin{equation}
\beta_1 = \langle \beta, Y_{110} \rangle Y_{110},
\end{equation}

\noindent
with $\langle \cdot, \cdot \rangle$ denoting the $L^2$ inner product on a round sphere of symmetry, and where our condition of axisymmetry removes from consideration the remaining eigensections $Y_{111}$ and $Y_{11-1}$.  Concretely, $Y_{110} = \sin^2\theta \Psi_{-1}$, per Appendix \ref{spinHarmonics}.

\begin{lemma}\label{beta1t}
The section $\beta_1$ is time-independent, with $\beta_{1,t} = 0.$
\end{lemma}

\begin{proof}
Projecting the linearized Einstein equation \eqref{R2phi}, 

$$
\sqrt{2}r\bar{\partial}\alpha + \beta_{t} = 0,
$$

\noindent
it suffices to show that

$$ (\bar{\partial}\alpha)_1 = 0.$$

The argument for the vanishing of  $(\bar{\partial}\alpha)_1$ proceeds as follows.  Letting $Y_{11m}$ be an eigensection of $\mathcal{E}(-2)$, the coefficients of $(\bar{\partial}\alpha)_1$ are given by

$$\langle \bar\partial\alpha, Y_{11m} \rangle = \langle \alpha, \bar\partial^{\dagger} Y_{11m} \rangle = -\langle \alpha, \partial Y_{11m} \rangle,$$

\noindent
and it suffices to show 

$$ \partial Y_{11m} = 0.$$

To this end, we note that

\begin{align*}
\slashed{\Delta}_{\mathcal{E}(-2)} Y_{11m} &= -Y_{11m},\\
\slashed{\Delta}_{\mathcal{E}(-2)} Y_{11m} &= 2\bar\partial \partial Y_{11m} - Y_{11m}.
\end{align*}

Taken together, the two imply

$$ \langle \slashed{\Delta}_{\mathcal{E}(-2)} Y_{11m}, Y_{11m} \rangle = -|Y_{11m}|^2 = \langle 2\bar\partial \partial Y_{11m}, Y_{11m} \rangle - |Y_{11m}|^2,$$

\noindent
such that

$$ \langle \bar\partial \partial Y_{11m}, Y_{11m} \rangle = 0.$$

Rewriting as 

$$\langle \bar\partial \partial Y_{11m}, Y_{11m} \rangle = \langle \partial Y_{11m}, \bar\partial^{\dagger} Y_{11m} \rangle = -\langle \partial Y_{11m}, \partial Y_{11m} \rangle,$$

\noindent
we see that $\partial Y_{11m} = 0,$ and the proof is complete.
\end{proof}

\begin{theorem}
Assuming $\beta$ satisfies a suitable decay condition at infinity, as in any asymptotically flat solution $(\alpha, \beta, \gamma)$, the lowest mode $\beta_1$ can be made to vanish by the addition of a suitable linearized Kerr solution.
\end{theorem}

\begin{proof}
From \eqref{beta1}, $\beta_1$ satisfies

\begin{equation}
\frac{-r^2}{\Delta}\beta_{1,tt} + \frac{\Delta}{r^2}\beta_{1,rr} + \frac{2r-2M}{r^2}\beta_{1,r} = \frac{2}{r^2}(1-\frac{4M}{r}) \beta_1,
\end{equation}

\noindent
or, in light of the previous lemma,

\begin{equation}
\frac{\Delta}{r^2}\beta_{1,rr} + \frac{2r-2M}{r^2}\beta_{1,r} = \frac{2}{r^2}(1-\frac{4M}{r}) \beta_1.
\end{equation}

The general solution to the radial ODE above is given by 

\begin{equation}
\beta_1 = \left[\frac{C_{1}}{r^2} + C_{2}\frac{(12M^2r + 3Mr^2 + r^3 + 24M^3\log(r-2M))}{3r^2}\right] \sin^2\theta \Psi_{-1}.
\end{equation}

However, by requiring that our linearized solutions to \eqref{eq: vacuumEinstein} obey an asymptotic flatness condition, we can rule out the second term, so that

\begin{equation}
\beta_1 = \frac{C_{1}}{r^2}\sin^2\theta \Psi_{-1}.
\end{equation}

As the linearized Kerr solution has $\beta = \frac{-6Ma^{(1)}}{r^2}\sin^2\theta\Psi_{-1}$ \eqref{KerrData}, we can always normalize by the addition of a suitable such linearized Kerr to ensure $\beta_{1} = 0$ throughout the exterior. 

\end{proof}

\section{Analysis of $\beta_{\ell > 1}$}

In light of the previous analysis, we can always account for the lowest mode of $\beta$ by addition of a suitable linearized Kerr.  We suppose $\beta$ that is supported away from this first eigenvalue, i.e. $\beta = \beta_{\ell > 1}$.

Away from the lowest mode, the analysis is entirely analogous to that conducted for $\alpha$.  The key observation is that, with the spectrum having the form described above, $\beta$ satisfies the Poincar\'{e} inequality

\begin{equation}\label{PoincareBeta}
\int_{S^2(r)} |\tilde{\slashed{\nabla}}\beta|^2 \geq \frac{5}{r^2} \int_{S^2(r)} |\beta|^2.
\end{equation}

This estimate allows for similar multipliers and estimates to those that applied to $\alpha$.  We recount these within the stress-energy formalism below.

\subsection{Stress-Energy Formalism}

Associated with our wave equation is the stress-energy tensor

\begin{equation}
T_{ab}[\beta]:= \frac{1}{2}(\slashed{\nabla}{a}\beta \slashed{\nabla}_{b}\bar{\beta} + \slashed{\nabla}_{a} \bar{\beta} \slashed{\nabla}_{b} \beta) - \frac{1}{2}g_{ab}(\slashed{\nabla}^{c}\beta\slashed{\nabla}_{c}\bar{\beta} + W|\beta|^2).
\end{equation}

The presence of a potential and our bundle setting introduce a nontrivial divergence:

\begin{equation}
\nabla^{a}T_{ab} = \frac{-1}{2}\nabla_{b}W|\beta|^2 + \frac{1}{2}\slashed{\nabla}^{a}\beta[\slashed{\nabla}_{a},\slashed{\nabla}_{b}]\bar{\beta} + \frac{1}{2}\slashed{\nabla}^{a}\bar{\beta}[\slashed{\nabla}_{a},\slashed{\nabla}_{b}]\beta.
\end{equation}

\noindent
Again, we remark that the commutator $[\slashed{\nabla}_{a},\slashed{\nabla}_{b}]$ vanishes upon application of the multipliers $T^{b}, N^{b}, X^{b}, Z^{b}$ under consideration.

\subsection{Degenerate Energy}

As in the analysis of $\alpha$, the Killing field $T$ has trivial density $K^{T}$, and we obtain a conservation law

\begin{equation}
E^{T}_{\beta}(\tau) = E^{T}_{\beta}(\tau'),
\end{equation}

\noindent
where the $T$-energy is defined just as before:

\begin{equation}
E^{T}_{\beta}(\tau) := \int_{\{ t = \tau \}} J^{T}_{a}\eta^{a}.
\end{equation}

 Appearing in the energy $E^{T}(\tau)$ is a potential term $W|P|^2$, seemingly quite bad, but it is more than accounted for by our Poincar\'{e} inequality (\ref{PoincareBeta}).  In particular, the conserved $T$-energy proves to be non-negative, but degenerate, just as before \eqref{eq: ConservationLaw}.

\subsection{Non-degenerate Energy and Uniform Boundedness}

The red-shift argument carries over with only small modifications.  In particular, the red-shift multiplier $N$, used earlier for $\alpha$, is seen to satisfy in our case the pointwise estimate

$$ K^{N} \geq { c J^{N}_{a}N^{a}},$$

\noindent
for some small $c>0$, near the event horizon.

The argument is much the same as with $\alpha$.  First, we note that the pointwise estimate

$$ T_{ab}\nabla^{a}N^{b} \geq{ c J^{N}_{a}N^{a}},$$

\noindent
continues to hold, being simply a consequence of non-trivial surface gravity $\kappa$ and tensorial algebra on $T_{ab}$.  In addition, the divergence term

$$\nabla^{a}T_{ab}N^{b} = \frac{-1}{2}N(W)|\beta|^2 \geq 0,$$

as with earlier study of $\alpha$.  Taken together, the two yield the desired estimate

$$ K^{N} \geq {T_{ab}\nabla^{a}N^{b}} \geq{c J^{N}_{a}N^{a}}.$$

Extending to the exterior, we construct a strictly timelike multiplier $N$, satisfying the estimates
 
\begin{align*}
K^{N} &\geq{c J^{N}_{a}N^{a}},\ &&\textup{for}\ 2M\leq r \leq r_0,\\
J^{N}_{a}N^{a} &\sim J^{T}_{a}T^{a}, \ &&\textup{for}\ r_0 \leq r \leq R_0,\\
|K^{N}| &\leq  C|J^{T}_{a}T^{a}|, \ &&\textup{for}\ r_0 \leq r \leq R_0,\\
N &= T, \ &&\textup{for}\ r\geq R_0,
\end{align*}
 
\noindent
for some fixed $2M < r_0 < R_0 < \infty$.

Defining the $N$-energy just as before,

\begin{equation}
E^{N}_{\beta}(\Sigma_{\tau}) := \int_{\{t_{*} = \tau\}} J^{N}_{a}\eta^{a},
\end{equation}

\noindent
we note that, although the flux $J^{N}_{a}\eta^{a}$ cannot be shown to be non-negative pointwise, application of our Poincar\'{e} inequality (\ref{PoincareBeta}) gives a positive-definite energy $E^{N}_{\beta}(\Sigma_{\tau})$.  With the above control on the density term, the previous red-shift argument goes through unchanged, and we have uniform boundedness of energy analogous to Theorem \ref{EnergyBoundA}.

\begin{theorem}\label{EnergyBoundB}

 Suppose $\beta$ is a solution of (\ref{eq: RW2}), smooth and compactly supported on the space-like hypersurface $\{ t_{*} = 0\}$.  Further, assume that $\beta$ is supported away from the first eigenspace; i.e., $\beta = \beta_1$.  For $\tau \geq 0$, $\beta$ satisfies the uniform energy estimate
 
 \begin{equation}
 E^{N}_{\beta}(\Sigma_{\tau}) \leq CE^{N}_{\beta}(\Sigma_0).
\end{equation}

\end{theorem}

Likewise, commutation by $T$ and by the $\Omega_{i}$, along with the Sobolev embeddings, carry over to give uniform boundedness of $\beta$ in terms of initial energy, as in Theorem \ref{PointwiseBoundA}.

\begin{theorem}\label{PointwiseBoundB}
For $\beta$ satisfying the hypotheses of Theorem \ref{EnergyBoundB} above, we have the pointwise bound

\begin{equation}
||\beta||_{C^{0}(\mathcal{D})}  \leq {C\left(E^{N}_{\beta}(\Sigma_0) + \sum_{i =1}^{3}E^{N}_{\Omega_{i}\beta}(\Sigma_0) + \sum_{i,j = 1}^{3} E^{N}_{\Omega_{i}\Omega_{j}\beta}(\Sigma_0)\right)},
\end{equation}

\noindent
with $\mathcal{D} = J^{+}(\Sigma_0)\cap{J^{-}(\mathcal{I}^{+})}$ being the exterior spacetime region foliated by the non-negative $t_{*}$-slices.
\end{theorem}

The above theorem uses again the shorthand

\begin{align}
T\beta &:= \mathcal{L}_{T}\beta,\\
\Omega_{i}\beta &:= \mathcal{L}_{\Omega_{i}}\beta,\\
\Omega_{i}\Omega_{j}\beta &:= \mathcal{L}_{\Omega_{i}}\mathcal{L}_{\Omega_{j}}\beta.
\end{align}

\subsection{The Morawetz Multiplier $X$}

As it turns out, the very same Morawetz multiplier which applied to $\alpha$ works just as well for $\beta$.  We remind the reader that, letting $X = f(r)\partial_{r_{*}},$ with $f$ a general radial function, and with weight function $\omega^{X} = f' + 2f\frac{1-\mu}{r}$, we have the weighted energy current

\begin{equation}
 J^{X,\omega^{X}}_{a} := J^{X}_{a} + \frac{1}{4}\omega^{X}\nabla_{a}|\beta|^2 -\frac{1}{4}\nabla_{a}\omega^{X}|\beta|^2, 
 \end{equation}

\noindent
with weighted density

\begin{equation}
K^{X,\omega^{X}} := K^{X} + \frac{1}{4}\omega^{X}\Box|\beta|^2 -\frac{1}{4} \Box \omega^{X} |\beta|^2.
\end{equation}

The weighted density is calculated to be

\begin{multline}
K^{X,\omega^{X}} =  \frac{f'}{1-\mu}|\slashed{\nabla}_{r_{*}}\beta|^2 + \frac{f}{r}(1-\frac{3M}{r})|\tilde{\slashed{\nabla}}\beta|^2\\
+ \left[\frac{-Mf}{r^2}W -\frac{1}{2}fW' -\frac{1}{4}\Box \omega^{X} \right]|\beta|^2.
\end{multline}

With the choice $f := (1-\frac{3M}{r})(1+\frac{M}{r})^2$, as before, we find positive radial and angular terms, and a badly behaved base term, negative in a spatially compact region away from the photon sphere.  Integrated over a spherically symmetric spacetime region, say $\{ \tau' \leq t \leq \tau \}$, we can apply our Poincar\'{e} inequality (\ref{PoincareBeta}), borrow from the good angular term, and find just as before

\begin{multline}
\int_{\{\tau' \leq{t}\leq{\tau}\}} K^{X,\omega^{X}}\\
\geq{c\int_{\tau'}^{\tau}\int_{2M}^{\infty}\int_{S^2}\left[\frac{1}{r^2}|\slashed{\nabla}_{r*}\beta|^2 + \frac{1}{r^3}|\beta|^2 + \frac{(r-3M)^2}{r^3}|\tilde{\slashed{\nabla}}\beta|^2\right]r^2d\sigma dr dt.}
\end{multline}
 
Likewise, estimating the boundary terms by the conserved $T$-energy $E^{T}_{\beta}(\tau)$ is straightforward, allowing for an estimate on the degenerate spacetime integral above (regardless of choice of positive $\tau'$ and $\tau$) in terms of initial data.  Summarizing, we have the analog of Lemma \ref{DegIntegratedDecayA}:
 
\begin{lemma}\label{DegIntegratedDecayB}
Suppose $\beta$ is a solution of (\ref{eq: RW2}), supported away from the first eigenspace; i.e., $\beta = \beta_1$.  Further, assume $\beta$ is smooth and compactly supported on the time slice $\{ t = 0 \}$.  Then for any $0 \leq{\tau'}\leq\tau$, $\beta$ satisfies the degenerate integrated decay estimate
  
\begin{multline}
\int_{\tau'}^{\tau}\int_{2M}^{\infty}\int_{S^2}\left[\frac{1}{r^2}|\slashed{\nabla}_{r*}\beta|^2 + \frac{1}{r^3}|\beta|^2 + \frac{(r-3M)^2}{r^3}|\tilde{\slashed{\nabla}}\beta|^2\right]r^2d\sigma dr dt\\
\leq{C\int_{\{\tau' \leq t \leq \tau\}} K^{X,\omega^{X}}} \leq {C\int_{\{ t = \tau'\}} J^{T}_{a}\eta^{a}} = CE^{T}_{\beta}(\tau'),
\end{multline}
 
\noindent
with the spacetime term degenerating at the event horizon $r = 2M$, at the photon sphere $r = 3M$, and at infinity.
 
\end{lemma}
 
 \subsection{The $Z$ Multiplier}
 
Next we outline the use of the $Z$ multiplier in our setting, concluding with a statement of uniform decay.  Recall that $Z$ is defined as
 
 \begin{equation}
  Z := u^2 \partial_{u} + v^2 \partial_{v} = \frac{1}{2}(t^2 + r_{*}^2)\partial_{t} + tr_{*}\partial_{r_{*}}.
  \end{equation}
 
 Again, we define a weighted energy current
 
 \begin{equation}
J^{Z,\omega^Z}_{a} := J^Z_{a}+\frac{1}{4}\omega^Z\nabla_{a} |\beta|^2-\frac{1}{4}|\beta|^2\nabla_{a}\omega^Z,
\end{equation}

\noindent
with weighted density

\begin{align}
\begin{split}
K^{Z,\omega^Z}&= t|\tilde{\slashed \nabla} \beta|^2\left(-1-\frac{\mu r_{*}}{2r}+\frac{r_{*}(1-\mu)}{r}\right)-\frac{1}{4}\Box\omega^Z |\beta|^2\\
&+\left(\frac{tr_{*}}{r}(1-\mu)W-\frac{1}{2}W\nabla_{b} Z^{b}-\frac{1}{2}Z(W) \right)|\beta|^2,\\
\end{split}
\end{align}

\noindent
where $\omega^{Z} = \frac{2tr^*}{r}(1-\mu)$.

From the scalar case, it is well-known that the first two coefficients in the weighted density are positive near the event horizon and near infinity.  However, the remaining coefficient only has positivity near infinity, and indeed approaches negative infinity at a logarithmic rate near the event horizon.  Integrating over a spherically symmetric spacetime region and borrowing from both of the other two terms (applying our Poincar\'{e} inequality (\ref{PoincareBeta}) yet again), we have positivity near both the event horizon and near infinity.  That is, we can fix $2M < r' < R' < \infty$ such that

\begin{align}
\begin{split}
\int_{S^2} K^{Z,\omega^Z}&\geq 0\ \textup{as}\ r\leq r'\ \textup{or}\ r\geq R',\\
|K^{Z,\omega^Z}|&\leq Ct\left(|\beta|^2+|\tilde{\slashed\nabla} \beta|^2\right)\ \textup{as}\ r'\leq r\leq R'.
\end{split}
\end{align}

With this estimate in place, we next control the unweighted energy flux by the weighted flux; defining

\begin{align}
E^{Z}_{\beta}(\tau) &:= \int_{\{t = \tau \}} J^{Z}_{a}\eta^{a},\\
E^{Z,\omega^{Z}}_{\beta}(\tau) &:= \int_{\{t = \tau\}} J^{Z,\omega^{Z}}_{a}\eta^{a},
\end{align}

\noindent
we show the analog of \eqref{eq: Zcomparison}:

\begin{equation}\label{eq: Zcomparison2}
E^{Z,\omega^{Z}}_{\beta}(\tau) \geq c E^{Z}_{\beta}(\tau).
\end{equation}

The approach is exactly the same as before: rewrite the $Z$-energies in terms of the scaling operator $S$ and $\underline{S}$ (\ref{Soperators}).  Applying integration by parts to the weight term

$$ \frac{1}{4}\omega^Z\nabla_{a} |\beta|^2-\frac{1}{4}|\beta|^2\nabla_{a}\omega^Z,$$

\noindent
we can factor in terms of these operators and apply the Poincar\'{e} inequality (\ref{PoincareBeta}) to obtain the desired estimate.

The remainder of the proof follows identically that applied to $\alpha$.  Namely, we prove a local decay estimate analogous to Lemma \ref{localized}, and use bootstrapping to prove boundedness of $E_{\beta}^{Z,\omega^{Z}}(\tau)$, hence $E_{\beta}^{Z}(\tau)$, as in Lemma \ref{weightedZbound}.  This uniform bound on the unweighted $Z$-energy yields statements about quadratic decay of the degenerate $E_{\beta}^{T}(\tilde{\Sigma}_{\tau})$ and the non-degenerate $E_{\beta}^{N}(\tilde{\Sigma}_{\tau})$, as in Theorem \ref{NondegIntegratedDecayA}.  Finally, application of appropriate Sobolev embeddings and commutation with the usual Killing fields gives a statement of pointwise decay, analogous to Theorem \ref{pointwisedecayA}.

\begin{theorem}\label{NondegIntegratedDecayB}
Suppose $\beta$ is a solution of (\ref{eq: RW2}), supported away from the first eigenspace; i.e., $\beta = \beta_1$.  Further, assume $\beta$ is smooth and compactly supported on the time slice $\{ t = 0 \}$.  Defining the initial energy

\begin{equation}
E_1[\beta] := \sum_{(m)\leq 3} \int_{\{t = 0 \}} r^2J^{N}_{a}[\Omega^{(m)}\beta]\eta^{a},
\end{equation}

\noindent
the solution $\beta$ satisfies the integrated decay estimate

\begin{equation}
E^{N}_{\beta}(\tilde{\Sigma}_{\tau})\leq CE_1[\beta]\tau^{-2},
\end{equation}

\noindent
uniformly on the hypersurfaces $\tilde{\Sigma}_{\tau}$ (\ref{decayFoliation}) defined above.
\end{theorem}

\begin{theorem}\label{pointwisedecayB}
Suppose $\beta$ is a solution of (\ref{eq: RW2}), satisfying the same assumptions as (\ref{NondegIntegratedDecayB}).  Defining the energy 

\begin{equation}
E_2[\beta] := \sum_{(m)\leq6} \int_{\{t = 0 \}} r^2J^{N}_{a}[\Omega^{(m)}\beta]\eta^{a},
\end{equation}

\noindent
the solution $\beta$ satisfies the uniform decay estimate 

\begin{equation}
\sup_{\tilde{\Sigma}_{\tau}} \beta \leq C\sqrt{E_2[\beta]}\tau^{-1}
\end{equation}

\noindent
on the family of hypersurfaces $\tilde{\Sigma}_{\tau}$ (\ref{decayFoliation}) specified above.
\end{theorem}

\section{Analysis of $\gamma$ and Proof of the Main Theorem}

Again, recall the governing equations

\begin{align}
R_{t\phi}^{(1)} = 0 &\rightarrow (r^2\beta)_{r} + \frac{r^4}{\Delta}\sqrt{2}r\bar{\partial}\gamma = 0,\\
R_{r\phi}^{(1)} = 0 &\rightarrow \sqrt{2}r\bar{\partial}\alpha + \beta_{t},\\
R_{\theta\phi}^{(1)} = 0 &\rightarrow (r^2\alpha)_{r} -\frac{r^4}{\Delta}\gamma_{t} = 0,\\
d^2\zeta^{(1)} = 0 &\rightarrow \frac{1}{r^2}\sqrt{2}r \partial \beta - \gamma_{r} + \frac{r^2}{\Delta}\alpha_{t} = 0.
\end{align}

Further, we assume that the lowest mode of $\beta$ has been accounted for by the addition of a suitable linearized Kerr, such that $\beta = \beta_1$.

Expressed in terms of the Regge-Wheeler coordinates $t$ and $r_{*}$, we have

\begin{align}
\gamma_{t} &= \alpha_{r_{*}} + \frac{2\Delta}{r^3}\alpha,\\
\gamma_{r_{*}} &= \alpha_{t} + (1-\mu)\frac{1}{r^2}\sqrt{2}r\partial{\beta}.
\end{align}

Let $R$ be the unit radial vector on $\tilde{\Sigma}_{\tau}$, where we recall the specification of $\tilde{\Sigma}_{\tau}$ (\ref{decayFoliation}).  Just as before, an estimate on the radial derivative $|R\gamma|$ is crucial in subsequent applications of Sobolev embedding.  

Near the horizon, 

\begin{align*}
R &\sim \frac{1}{1-\mu}(\frac{-2M}{r}\partial_{t} + \partial_{r_{*}}),\\
\eta_{\tilde{\Sigma}_{\tau}} &\sim \frac{1}{1-\mu}(\partial_{t} -\frac{2M}{r} \partial_{r_{*}}).
\end{align*}

Hence

\begin{align*}
R\gamma &\sim \eta_{\tilde{\Sigma}_{\tau}}\alpha -\frac{4M}{r^2}\alpha + \sqrt{2}\partial \beta\\
|R\gamma|^2 &\leq C\left(|\eta_{\tilde{\Sigma}_{\tau}}\alpha|^2 + |\alpha|^2 + |\tilde{\slashed{\nabla}}\beta|^2\right).
\end{align*}

Near null infinity, we have

\begin{align*}
&R  :=t\partial_r+r\partial_t,\\
&\eta_{\tilde{\Sigma}_{\tau}}\sim r\partial_r+t\partial_t,
\end{align*}

\noindent
so that

$$
|R\gamma|^2 \leq C\left(|\eta_{\tilde{\Sigma}_{\tau}}\alpha|^2 + |\alpha|^2 + |\tilde{\slashed{\nabla}}\beta|^2\right),
$$

\noindent
as well.

Note also that the governing equations above commute with the angular Killing fields $\Omega_{i}$, so that 

$$|R\Omega^{(m)}\gamma|^2 \leq C\left(|\eta_{\tilde{\Sigma}_{\tau}}\Omega^{(m)}\alpha|^2 + |\Omega^{(m)}\alpha|^2 + |\tilde{\slashed{\nabla}}\Omega^{(m)}\beta|^2\right),$$

\noindent
for a multi-index $(m)$.

With $\rho$ a geodesic radial coordinate on $\tilde{\Sigma}_{\tau}$, normalized to $\rho=1$ on the horizon, we compute

\begin{align*}
|\gamma|\leq &\int_1^\infty |R\gamma|d\rho \leq C \left(\int_1^\infty |R\gamma|^2\rho^2 d\rho\right)^{1/2}\\
     \leq &C\left( \int_1^\infty d\rho\int_{S^2}d\sigma \rho^2\left( |R\gamma|^2+|R\Omega \gamma|^2+|R\Omega^2\gamma|^2 \right) \right)^{1/2}\\
     \leq &C\left(\int_{\tilde{\Sigma}_{\tau}} \frac{\rho^2}{r^2}(|\eta_{\tilde{\Sigma}_{\tau}}\alpha|^2 + |\alpha|^2 + |\tilde{\slashed{\nabla}}P|^2 + \hdots)\right)^{1/2}\\
     \leq &C\sqrt{E_2}\tau^{-1}.
\end{align*}

With this estimate, as well as those on $\alpha$ and $\beta$, we have our main theorem:

\begin{theorem}
Suppose $(\alpha, \beta, \gamma)$ is an axial solution of the linearized Einstein vacuum equations about the Schwarzschild spacetime.  Assume, moreover, that each of the three quantities is smooth and compactly supported on $\{ t = 0 \}$.  Then the following hold:

\begin{enumerate}

\item with the addition of a suitable linearized Kerr solution, the lowest spherical mode $\beta_1$ can be made to vanish, yielding the normalization $\beta = \beta_{\ell > 1}$,

\item each of the three quantities in the normalized solution $(\alpha, \beta, \gamma)$ decay, both in energy and pointwise, through the foliation $\tilde{\Sigma}_{\tau}.$

\end{enumerate}
\end{theorem}

We remark that all the above theorems hold with rougher initial data, with regularity corresponding to the completion of the weighted Sobolev norms specified by the energies (\ref{energies}) and with appropriate decay at infinity.

\section{Remark on the Non-linear Situation}

Although a useful way to capture the effect of angular momentum at the linear level, the axial family is seen to be entirely trivial when viewed in the non-linear setting.

\begin{theorem}
Suppose $(M, g)$ is an axisymmetric spacetime, axial with respect to a reference Schwarzschild spacetime $(\mathcal{M},g_{M})$, with metric of the form

$$ g = -\left(1-\frac{2M}{r}\right)dt^2 + \left(1-\frac{2M}{r}\right)dr^2 + r^2d\theta^2 + r^2\sin^2\theta(d\phi - \omega dt - q_2 dr - q_3 d\theta)^2,$$

\noindent
on the coordinate patch  $(t,r,\theta,\phi)$ with $t\in{\R}, r > 2M, (\theta,\phi)\in{S^2}$.

Further, suppose that $(M,g)$ satisfies the vacuum Einstein equations.  Then $(M,g)$ is isometric to Schwarzschild.

\end{theorem}

\begin{proof}
Per Chandrasekhar \cite{Chandra1}, for a generic axisymmetric metric placed into our standard form,

\begin{align}
\begin{split}
G_{tt}(g) &= -e^{-2\mu_2}\left[(\psi + \mu_3)_{rr} + \psi_{r}(\psi - \mu_2 + \mu_3)_{r} + \mu_{3,r}(\mu_3 - \mu_2)_{r}\right]\\
 &- e^{-2\mu_3}\left[(\psi + \mu_2)_{\theta \theta} + \psi_{\theta}(\psi - \mu_3 + \mu_2)_{\theta} + \mu_{2,\theta}(\mu_2 - \mu_3)_{\theta}\right]\\
& + e^{-2\nu}\left[\psi_{t}(\mu_2 + \mu_3)_{t} + \mu_{3,t}\mu_{2,t}\right]\\
& -\frac{1}{4}e^{2\psi - 2\nu}\left[e^{-2\mu_2}Q_{20}^2 + e^{-2\mu_3}Q_{30}^2\right] - \frac{1}{4}e^{2\psi - 2\mu_2 - 2\mu_3}Q_{23}^2.
\end{split}
\end{align}

However, given the form of our metric, the first three lines above are nothing other than a component of the Einstein tensor of our reference Schwarzschild.  That is,

\begin{align*}
G(g_{M})_{tt} &= e^{-2\mu_2}\left[(\psi + \mu_3)_{rr} + \psi_{r}(\psi - \mu_2 + \mu_3)_{r} + \mu_{3,r}(\mu_3 - \mu_2)_{r}\right]\\
 &- e^{-2\mu_3}\left[(\psi + \mu_2)_{\theta \theta} + \psi_{\theta}(\psi - \mu_3 + \mu_2)_{\theta} + \mu_{2,\theta}(\mu_2 - \mu_3)_{\theta}\right]\\
& + e^{-2\nu}\left[\psi_{t}(\mu_2 + \mu_3)_{t} + \mu_{3,t}\mu_{2,t}\right] = 0\\
\end{align*}

Applying our vacuum condition on $(M,g)$, we find

\begin{equation}
G_{tt}(g) = -\frac{1}{4}e^{2\psi - 2\nu}\left[e^{-2\mu_2}Q_{20}^2 + e^{-2\mu_3}Q_{30}^2\right] - \frac{1}{4}e^{2\psi - 2\mu_2 - 2\mu_3}Q_{23}^2 = 0.
\end{equation}

The expression is manifestly non-negative, vanishing only when $Q_{02} = Q_{03} = Q_{23} = 0$.  As these appears as coefficients in the two-form

$$
d\zeta = Q_{02} dt\wedge dr + Q_{03} dt \wedge d\theta + Q_{23} dr\wedge d\theta,
$$

\noindent
the one-form $\zeta$ is seen to be closed, hence exact.  Letting $F(t,r,\theta)$ be an associated scalar potential, 
$dF = \zeta = \omega dt + q_2 dr + q_3 d\theta,$ we take the change of coordinate $\tilde{\phi} = \phi - F(t,r,\theta)$ and find

$$ g = -\left(1-\frac{2M}{r}\right)dt^2 + \left(1-\frac{2M}{r}\right)dr^2 + r^2d\theta^2 + r^2\sin^2\theta d\tilde{\phi}^2,$$

\noindent
relating $(M,g)$ to the reference Schwarzschild $(\mathcal{M},g_{M})$.

\end{proof}








\section{Acknowledgements}
The authors would like to thank the National Center for Theoretical Science of National Taiwan University, where this research was initiated, for their warm hospitality.  As well, the authors thank their advisor, Mu-Tao Wang, for his patient guidance and support, without which this work would not have been possible.  We also thank Ye-Kai Wang for many stimulating conversations.
\appendix

\section{The Cotton-Darboux Theorem}\label{CottonDarboux}

The metric expression 

\begin{equation}
g_{\epsilon} = -e^{2\nu}(dx^0)^2 + e^{2\mu_2}(dx^2)^2 + e^{2\mu_3}(dx^3)^2 + e^{2\psi}(d\phi -\omega dx^0 - q_2 dx^2 - q_3 dx^3)^2,
\end{equation}

\noindent
adopted in this paper, which could well be thought of as an ansatz, is in fact a generic coordinate condition for axisymmetric metrics satisfying a mild causality condition.  A proof of this genericity is implied by the Cotton-Darboux theorem, which we recount below. 

Let $(M,g)$ be a spacetime with axisymmetric metric $g$, satisfying the condition of stable causality (see Wald \cite{Wald}).
Choosing coordinates $(x^0, x^1, x^2, \phi)$, where $\phi$ is an axisymmetric coordinate, we consider the contravariant form of the metric 

$$ h = h^{ij}\partial_{i}\partial_{j},$$

\noindent
where $i,j = 0,1, 2$, defined on the level set $\{\phi = c\}$.  


\begin{theorem}(Cotton-Darboux)
A non-degenerate, pseudo-Riemannian metric $h$, expressed in local coordinates $(x^0, x^1, x^2)$ as

$$ h = h^{ij}\partial_{i}\partial_{j},$$

\noindent
can be locally diagonalized by a suitable change of coordinates.
\end{theorem}

\begin{proof}

Let $f(x^0, x^1, x^2)$ be a function with $|\nabla f|^2 = h^{ij}f_{i}f_{j} \neq{0}$.  In particular, the existence of such a function in our spacetime setting, with $\nabla f$ being strictly time-like, is implied by stable causality.  Fixing the level set $\{ f = 0 \}$, we define the $\hat{x}^0$ as the geodesic coordinate associated with the geodesics passing through $\{ f = 0\}$, in the direction $\nabla f$.  Choosing coordinates $\hat{x}^1, \hat{x}^2$ on the surface $\{ f = 0 \}$, we propagate them along the direction $\nabla f$ to obtain a local coordinate system $(\hat{x}^0, \hat{x}^1, \hat{x}^2)$.  In these new coordinates, our metric has the form

$$ h = -\partial_{\hat{0}}\partial_{\hat{0}} + h^{\alpha\beta}\partial_{\alpha}\partial_{\beta},$$

\noindent
for $\alpha, \beta = \hat{1},\hat{2}$.

Next, we seek a change of coordinates,

$$ \bar{x}^{i} = \kappa^{i}(\hat{x}^0,\hat{x}^1,\hat{x}^2),$$

\noindent
yielding a diagonal form of our metric.  For this to be the case, the transformation $\kappa = (\kappa^0, \kappa^1, \kappa^2)$ must satisfy

\begin{align}
\begin{split}
\frac{\partial{\kappa^1}}{\partial \hat{x}^0}\frac{\partial{\kappa^2}}{\partial\hat{x}^0} &= h^{\alpha\beta}\frac{\partial \kappa^1}{\partial \hat{x}^{\alpha}}\frac{\partial \kappa^2}{\partial \hat{x}^{\beta}} =: K^0,\\
\frac{\partial{\kappa^2}}{\partial \hat{x}^0}\frac{\partial{\kappa^0}}{\partial\hat{x}^0} &= h^{\alpha\beta}\frac{\partial \kappa^2}{\partial \hat{x}^{\alpha}}\frac{\partial \kappa^0}{\partial \hat{x}^{\beta}} =: K^1,\\
\frac{\partial{\kappa^0}}{\partial \hat{x}^0}\frac{\partial{\kappa^1}}{\partial\hat{x}^0} &= h^{\alpha\beta}\frac{\partial \kappa^0}{\partial \hat{x}^{\alpha}}\frac{\partial \kappa^1}{\partial \hat{x}^{\beta}} =: K^2,
\end{split}
\end{align}

\noindent
where we have used suggestive notation $K^i$ for the RHS, expressible in terms of spatial data.  The above equations imply

\begin{align}
\begin{split}
\frac{\partial \kappa^0}{\partial \hat{x}^0} &= \left(\frac{K^1 K^2}{K^0}\right)^{1/2},\\
\frac{\partial \kappa^1}{\partial \hat{x}^0} &= \left(\frac{K^2 K^0}{K^1}\right)^{1/2},\\
\frac{\partial \kappa^2}{\partial \hat{x}^0} &= \left(\frac{K^0 K^1}{K^2}\right)^{1/2}.
\end{split}
\end{align}

Specifying $\kappa = (\kappa^0, \kappa^1, \kappa^2)$ on the surface $\{ f = 0 \} = \{ \hat{x}^0 = 0\}$, the Cauchy-Kowalewski theorem gives local existence to the 
associated initial value problem, and hence a local coordinate system $(\bar{x}^0, \bar{x}^1, \bar{x}^2)$ diagonalizing the metric $h$.

\end{proof}

\section{Spin-Weighted Spherical Harmonics}\label{spinHarmonics}

We present a few details from the work of Eastwood and Tod \cite{EastwoodTod}, introducing spin-weighted spherical harmonics as an extension of the usual spherical harmonics.

On the holomorphic line bundle $\mathcal{E}(-2s)$ over $S^2$, we define the $L^2$ inner product

\begin{equation}
\langle f, g \rangle = \int_{S^2} f\bar{g}
\end{equation}

\noindent
where the measure on $S^2$ is determined by the natural Hermitian structure of $\mathcal{E}(-2s)$.  The associated inner-product space $L^{2}(s)$ is seen to be a Hilbert space.

It is well known that there exist ``spin-weighted" spherical harmonics on the bundle $\mathcal{E}(-2s)$, which we denote by $Y_{s\ell m}(\theta,\phi)$, for $\ell \geq |s|$ and $-\ell \leq m \leq \ell$.  These harmonics are eigensections of the associated Laplace-Beltrami operator, with

\begin{equation}
\slashed{\Delta}_{\mathcal{E}(-2s)} Y_{s\ell m} = (s^2-\ell (\ell + 1)) Y_{s\ell m},
\end{equation}

\noindent
reducing to the usual spherical harmonics for $s = 0$.  As in the case of spherical harmonics, it is well-known that the spin-weighted spherical harmonics characterize the spectrum of $\slashed{\Delta}_{\mathcal{E}(-2s)}$, with eigenvalues $\lambda = \ell(\ell + 1)-s^2$ for $\ell \geq |s|$.  Even more, we have

\begin{theorem}
The spin-weighted spherical harmonics $Y_{s\ell m}$ are a complete, orthonormal basis in the Hilbert space $L^{2}(s)$.
\end{theorem}

See \cite{EastwoodTod} for a proof. 

The spin-weighted spherical harmonics are further related by the holomorphic and anti-holomorphic differentials $\partial$ and $\bar\partial$.  In particular, beginning from the well-known spherical harmonics $Y_{0\ell m}$, we find

\begin{align*}\label{mapUP}
\partial^{s}Y_{0\ell m} &\sim Y_{s\ell m} \textup{, for}\ s\leq l,\\
\partial^{s} Y_{0\ell m} &= 0 \textup{, for}\ s > \ell,
\end{align*}

\noindent
where we use $\sim$ to indicate that the two differ by a scalar multiple.  In this way, we can construct the spin-weighted spherical harmonics using just the usual ones.

\section{Stress-Energy Formalism}\label{stressEnergy}

As in previous work on the subject, our analysis relies upon vector field multiplier methods, outlined below.  

Associated with a solution $\Upsilon$ of the wave equation $\slashed{\Box}_{\mathcal{L}(-s)}\Upsilon = 0$ is the stress-energy tensor:

\begin{equation}
T_{ab}[\Upsilon] = \frac{1}{2}(\slashed{\nabla}_a \Upsilon \slashed{\nabla}_b \bar{\Upsilon} + \slashed{\nabla}_a \bar{\Upsilon} \slashed{\nabla}_b \Upsilon) - \frac{1}{2}g_{ab}\slashed{\nabla}^{c}\Upsilon\slashed{\nabla}_{c}\bar{\Upsilon}.
\end{equation}

The essential idea behind the multiplier method is to act upon the stress-tensor $T_{ab}$ by a vector field $C^{b}$, applying Stokes' theorem to the resulting one-form.  That is, in a spacetime region $\mathcal{D}$, Stokes' theorem implies

\begin{equation}
\int_{\mathcal{\partial D}} T_{ab}C^{b}\eta^{a} = \int_{\mathcal{D}} \nabla^{a}(T_{ab}C^{b}),
\end{equation} 

\noindent
where $\eta^{a}$ is the normal vector to $\partial \mathcal{D}$, the boundary of $\mathcal{D}$.  The relation between these hypersurface boundary and spacetime "bulk" integrals is the foundation of our analysis.

For simplicity, we denote the contraction one-form, known as the energy current, by 

\begin{equation}
J^{C}_{a} = T_{ab}C^{b}.
\end{equation}  

Likewise, we denote the interior integrand, the energy density, by 

\begin{equation}
K^{C} = \nabla^{a}J^{C}_{a} = \nabla^{a}(T_{ab}C^{b}).
\end{equation}

The above allow us to rewrite Stokes' theorem in the concise form

\begin{equation}
\int_{\mathcal{\partial D}} J^{C}_{a}\eta^{a} = \int_{\mathcal{D}} K^{C}.
\end{equation} 

Note that the divergence term $K^{X}$ expands as

$$ K^{C} = \nabla^{a}(T_{ab}C^{b}) = (\nabla^{a}T_{ab})C^{b} + T_{ab}\nabla^{a}C^{b},$$ 

\noindent
where $T_{ab}$ is divergence-free in this simple case.  The second term involves the deformation tensor $\pi^{C}_{ab} = \nabla_{(a}C_{b)}$, expressible as the Lie derivative $\pi^{C}_{ab} = \frac{1}{2}(\mathcal{L}_{C}g)_{ab}.$  In the case where $C$ is Killing, we remark that the deformation tensor $\pi^{C}$ vanishes.

For more information on the vector field multiplier method, we refer the reader to \cite{Alinhac2}.

\section{Construction of the Red-Shift Multiplier $N$}\label{redShift}
We recount the coordinate independent construction of the red-shift multiplier $N$ given by Dafermos and Rodnianski in \cite{DRClay}, for the wave equation

$$ \slashed{\Box}_{\mathcal{L}(-s)} \Upsilon = 0.$$

To construct $N$, we begin on the event horizon $\mathcal{H}^{+}$, away from the bifurcation sphere.  Let $Y$ be a null transversal to the Killing field $T$, itself tangential on $\mathcal{H}^{+}$.  We specify $Y$ by

\begin{enumerate}
\item $Y$ is future-directed, with normalization $Y \cdot T = -2$,
\item $Y$ is invariant under $T$ and the $SO(3)$ symmetry,
\item On $\mathcal{H}^{+}$, $\nabla_{Y}Y = -\sigma(Y + T)$, for $\sigma \in \R$ as yet unchosen.
\end{enumerate}

Taking $E_{i}$ as tangential to the sphere, we calculate in the frame $\{ T, Y, E_1, E_2 \}$:

\begin{align*}
\nabla_{T}Y = & -\kappa Y \\
\nabla_{Y}Y = & -\sigma(T + Y)\\
\nabla_{E_{i}}Y = &h^{j}_{i}E_{j},
\end{align*}
 
\noindent
with $\kappa$ being the positive surface gravity on Schwarzschild spacetime, and with $h^{i}_{j}$ being the second fundamental form of the sphere $r = 2M$ (recall that we work on the event horizon) with respect to $Y$.
 
We compute
 
$$ K^{Y} = (\nabla^{a}T_{ab})Y^{b} + T_{ab}\nabla^{a}Y^{b}, $$
 
\noindent
with
 
$$ (\nabla^{a}T_{ab})Y^{b} =  0.$$
 
With an adequate choice of $\sigma$, we can ensure that 
 
$$ K^{Y} = T_{ab}\nabla^{a}Y^{b} \geq {c J^{T+Y}_{a}(T + Y)^{a}}.$$
  
Define $N = Y + T$, called the red-shift multiplier.  From the above, $N$ is revealed to be future-directed and time-like along the horizon $\mathcal{H}^{+}$, with 
 
$$ K^{N} = K^{Y} \geq{c J^{N}_{a}N^{a}}.$$
 
Extending to the exterior, we construct a strictly timelike multiplier $N$, satisfying the estimates
 
\begin{align*}
K^{N} &\geq{c J^{N}_{a}N^{a}},\ &&\textup{for}\ 2M\leq r \leq r_0,\\
J^{N}_{a} &\sim J^{T}_{a}, \ &&\textup{for}\ r_0 \leq r \leq R_0,\\
|K^{N}| &\leq  C|J^{T}_{a}T^{a}| \ &&\textup{for}\ r_0 \leq r \leq R_0,\\
N &= T \ &&\textup{for}\ r\geq R_0,
\end{align*}
 
\noindent
for some fixed $2M < r_0 < R_0 < \infty$.

\bibliographystyle{plain}
\bibliography{AxialStability}

\begin{thebibliography}{10}

\bibitem{Alinhac1}
S.~Alinhac.
\newblock Energy multipliers for perturbations of the {S}chwarzschild metric.
\newblock {\em Comm. Math. Phys.}, 288(1):199--224, 2009.

\bibitem{Alinhac2}
S.~Alinhac.
\newblock {\em Geometric analysis of hyperbolic differential equations: an
  introduction}, volume 374 of {\em London Mathematical Society Lecture Note
  Series}.
\newblock Cambridge University Press, Cambridge, 2010.

\bibitem{AnderssonBlue}
L.~Andersson and P.~Blue.
\newblock Hidden symmetries and decay for the wave equation on the {Kerr}
  spacetime.
\newblock {\em arXiv preprint}, 2009.
\newblock arXiv:0908.2265.

\bibitem{Bieri}
L.~Bieri and N.~Zipser.
\newblock {\em Extensions of the stability theorem of the {M}inkowski space in
  general relativity}, volume~45 of {\em AMS/IP Studies in Advanced
  Mathematics}.
\newblock American Mathematical Society, Providence, RI; International Press,
  Cambridge, MA, 2009.

\bibitem{BlueSterbenz}
P.~Blue and J.~Sterbenz.
\newblock Uniform decay of local energy and the semi-linear wave equation on
  {S}chwarzschild space.
\newblock {\em Comm. Math. Phys.}, 268(2):481--504, 2006.

\bibitem{Chandra2}
S.~Chandrasekhar.
\newblock On the equations governing the perturbations of the {S}chwarzschild
  black hole.
\newblock {\em Proc. Roy. Soc. (London) Ser. A}, 343:289--298, 1975.

\bibitem{Chandra1}
S.~Chandrasekhar.
\newblock {\em The mathematical theory of black holes}.
\newblock Oxford Classic Texts in the Physical Sciences. The Clarendon Press,
  Oxford University Press, New York, 1998.
\newblock Reprint of the 1992 edition.

\bibitem{ChristKlainerman}
D.~Christodoulou and S.~Klainerman.
\newblock {\em The global nonlinear stability of the {M}inkowski space},
  volume~41 of {\em Princeton Mathematical Series}.
\newblock Princeton University Press, Princeton, NJ, 1993.

\bibitem{DHR}
M.~Dafermos, G.~Holzegel, and I.~Rodnianski.
\newblock The linear stability of the {Schwarzschild} solution to gravitational
  perturbations.
\newblock {\em arXiV preprint}, 2016.
\newblock arXiv: 1601.06467.

\bibitem{DR}
M.~Dafermos and I.~Rodnianski.
\newblock The red-shift effect and radiation decay on black hole spacetimes.
\newblock {\em Comm. Pure Appl. Math.}, 62(7):859--919, 2009.

\bibitem{DRClay}
M.~Dafermos and I.~Rodnianski.
\newblock Lectures on black holes and linear waves.
\newblock In {\em Evolution equations}, volume~17 of {\em Clay Math. Proc.},
  pages 97--205. Amer. Math. Soc., Providence, RI, 2013.

\bibitem{DRR}
M.~Dafermos, I.~Rodnianski, and Y.~Shlapentokh-Rothman.
\newblock Decay for solutions of the wave equation on {K}err exterior
  spacetimes {III}: {T}he full subextremal case {$\vert a\vert <M$}.
\newblock {\em Ann. of Math. (2)}, 183(3):787--913, 2016.

\bibitem{EastwoodTod}
M.~Eastwood and P.~Tod.
\newblock Edth---a differential operator on the sphere.
\newblock {\em Math. Proc. Cambridge Philos. Soc.}, 92(2):317--330, 1982.

\bibitem{Smoller}
F.~Finster, N.~Kamran, J.~Smoller, and S.-T. Yau.
\newblock Decay of solutions of the wave equation in the {K}err geometry.
\newblock {\em Comm. Math. Phys.}, 264(2):465--503, 2006.

\bibitem{KayWald}
B.~S. Kay and R.~M. Wald.
\newblock Linear stability of {S}chwarzschild under perturbations which are
  nonvanishing on the bifurcation {$2$}-sphere.
\newblock {\em Classical Quantum Gravity}, 4(4):893--898, 1987.

\bibitem{Klainerman2}
S.~Klainerman.
\newblock Global existence for nonlinear wave equations.
\newblock {\em Comm. Pure Appl. Math.}, 33(1):43--101, 1980.

\bibitem{Klainerman1}
S.~Klainerman.
\newblock Uniform decay estimates and the {L}orentz invariance of the classical
  wave equation.
\newblock {\em Comm. Pure Appl. Math.}, 38(3):321--332, 1985.

\bibitem{Kruskal}
M.~D. Kruskal.
\newblock Maximal extension of {S}chwarzschild metric.
\newblock {\em Phys. Rev. (2)}, 119:1743--1745, 1960.

\bibitem{Lindblad}
H.~Lindblad and I.~Rodnianski.
\newblock The global stability of {M}inkowski space-time in harmonic gauge.
\newblock {\em Ann. of Math. (2)}, 171(3):1401--1477, 2010.

\bibitem{Luk}
J.~Luk.
\newblock Improved decay for solutions to the linear wave equation on a
  {S}chwarzschild black hole.
\newblock {\em Ann. Henri Poincar\'e}, 11(5):805--880, 2010.

\bibitem{Tataru1}
J.~Marzuola, J.~Metcalfe, D.~Tataru, and M.~Tohaneanu.
\newblock Strichartz estimates on {S}chwarzschild black hole backgrounds.
\newblock {\em Comm. Math. Phys.}, 293(1):37--83, 2010.

\bibitem{Moncrief}
V.~Moncrief.
\newblock Gravitational perturbations of spherically symmetric systems. {I}.
  {T}he exterior problem.
\newblock {\em Ann. Physics}, 88:323--342, 1974.

\bibitem{Morawetz}
C.~S. Morawetz.
\newblock The limiting amplitude principle.
\newblock {\em Comm. Pure Appl. Math.}, 15:349--361, 1962.

\bibitem{RW}
T.~Regge and J.~A. Wheeler.
\newblock Stability of a {S}chwarzschild singularity.
\newblock {\em Phys. Rev. (2)}, 108:1063--1069, 1957.

\bibitem{Speck}
J.~Speck.
\newblock The global stability of the {M}inkowski spacetime solution to the
  {E}instein-nonlinear system in wave coordinates.
\newblock {\em Anal. PDE}, 7(4):771--901, 2014.

\bibitem{Tataru2}
D.~Tataru and M.~Tohaneanu.
\newblock A local energy estimate on {K}err black hole backgrounds.
\newblock {\em Int. Math. Res. Not. IMRN}, (2):248--292, 2011.

\bibitem{Vishveshwara}
C.~V. Vishveshwara.
\newblock Stability of the {Schwarzschild} metric.
\newblock {\em Phys. Rev. D}, 1:2870--2879, 1970.

\bibitem{Wald}
R.~M. Wald.
\newblock {\em General relativity}.
\newblock University of Chicago Press, Chicago, IL, 1984.

\bibitem{Zerilli}
F.~J. Zerilli.
\newblock Effective potential for even-parity {Regge-Wheeler} gravitation
  perturbation equations.
\newblock {\em Phys. Rev. Lett.}, 24:737--738, 1970.

\end{thebibliography}

\end{document}